\documentclass[12pt]{article} 
\usepackage{amsfonts,amsmath,latexsym,amssymb,mathrsfs,amsthm,comment}
\usepackage{slashbox,color}
\usepackage{caption}

\let\OLDthebibliography\thebibliography
\renewcommand\thebibliography[1]{
  \OLDthebibliography{#1}
  \setlength{\parskip}{3pt}
  \setlength{\itemsep}{0pt plus 0.3ex}
}

\def\numberlikeadb{\global\def\theequation{\thesection.\arabic{equation}}}
\numberlikeadb
\newtheorem{theorem}{Theorem}[section]
\newtheorem{lemma}[theorem]{Lemma}

\newtheorem{corollary}[theorem]{Corollary}
\newtheorem{definition}[theorem]{Definition}
\newtheorem{proposition}[theorem]{Proposition}
\newtheorem{remark}[theorem]{Remark}

\usepackage{lscape}
\usepackage{caption}
\usepackage{multirow}
\begin{document}

\title{Wasserstein and Kolmogorov error bounds for variance-gamma approximation via Stein's method I}
\author{Robert E. Gaunt\footnote{School of Mathematics, The University of Manchester, Manchester M13 9PL, UK}}

\date{\today} 
\maketitle

\vspace{-5mm}

\begin{abstract}The variance-gamma (VG) distributions form a four parameter family that includes as special and limiting cases the normal, gamma and Laplace distributions.  Some of the numerous applications include financial modelling and approximation on Wiener space.  Recently, Stein's method has been extended to the VG distribution.  However, technical difficulties have meant that bounds for distributional approximations have only been given for smooth test functions (typically requiring at least two derivatives for the test function).  In this paper, which deals with symmetric variance-gamma (SVG) distributions, and a companion paper \cite{gaunt vgii}, which deals with the whole family of VG distributions, we address this issue.  In this paper, we obtain new bounds for the derivatives of the solution of the SVG Stein equation, which allow for approximations to be made in the Kolmogorov and Wasserstein metrics, and also introduce a distributional transformation that is natural in the context of SVG approximation.  We apply this theory to obtain Wasserstein or Kolmogorov error bounds for SVG approximation in four settings: comparison of VG and SVG distributions, SVG approximation of functionals of isonormal Gaussian processes, SVG approximation of a statistic for binary sequence comparison, and Laplace approximation of a random sum of independent mean zero random variables. 
\end{abstract}


\noindent{{{\bf{AMS 2010 Subject Classification:}}} Primary 60F05; 62E17

\section{Introduction}

\subsection{Overview of Stein's method for variance-gamma approximation}

The variance-gamma (VG) distribution with parameters $r > 0$, $\theta \in \mathbb{R}$, $\sigma >0$, $\mu \in \mathbb{R}$ has probability density function
\begin{equation}\label{vgdefn}p(x) = \frac{1}{\sigma\sqrt{\pi} \Gamma(\frac{r}{2})} \mathrm{e}^{\frac{\theta}{\sigma^2} (x-\mu)} \bigg(\frac{|x-\mu|}{2\sqrt{\theta^2 +  \sigma^2}}\bigg)^{\frac{r-1}{2}} K_{\frac{r-1}{2}}\bigg(\frac{\sqrt{\theta^2 + \sigma^2}}{\sigma^2} |x-\mu| \bigg), 
\end{equation}
where $x \in \mathbb{R}$ and the modified Bessel function of the second kind $K_\nu(x)$ is defined in Appendix \ref{appendix}. If the random variable $Z$ has density (\ref{vgdefn}), we write $Z\sim\mathrm{VG}(r,\theta,\sigma,\mu)$.  The support of the VG distributions is $\mathbb{R}$ when $\sigma>0$, but in the limit $\sigma\rightarrow 0$ the support is the region $(\mu,\infty)$ if $\theta>0$, and is $(-\infty,\mu)$ if $\theta<0$.  Alternative parametrisations are given in \cite{eberlein} and \cite{kkp01} (in which they use the name generalized Laplace distribution).  Distributional properties are given in \cite{gaunt vg} and Chapter 4 of the book \cite{kkp01}.  


The VG distribution was introduced to the financial literature by \cite{madan}.  Due to their semi-heavy tails, VG distributions are useful for modelling financial data \cite{madan2}; see the book \cite{kkp01} and references therein for an overview of the many applications.  The class of VG distributions contain many classical distributions as special or limiting cases, such as the normal, gamma, Laplace, product of zero mean normals and difference of gammas (see Proposition 1.2 of \cite{gaunt vg} for a list of further cases).  Consequently, the VG distribution appears in many other settings beyond financial mathematics \cite{kkp01}; for example, in alignment-free sequence comparison \cite{lippert, waterman}.  In particular, starting with the works \cite{gaunt thesis, gaunt vg}, Stein's method \cite{stein} has been developed for VG approximation.  The theory of \cite{gaunt thesis, gaunt vg} and the Malliavin-Stein method (see \cite{np12}) was applied by \cite{eichelsbacher} to obtain ``six moment" theorems for the VG approximation of double Wiener-It\^{o} integrals.  Further VG approximations are given in \cite{aaps17} and \cite{bt17}, in which the limiting distribution is the difference of two centered gamma random variables.

Introduced in 1972, Stein's method  \cite{stein} is a powerful tool for deriving distributional approximations with respect to a probability metric.  The theory for normal and Poisson approximation is particularly well established with numerous application in probability and beyond; see the books \cite{chen} and \cite{bhj92}.  There is active research into the development of Stein's method for other distributional limits (see \cite{ley} for an overview), and Stein's method for exponential and geometric approximation, for example, is now also well developed; see the survey \cite{ross}.  In particular, \cite{pekoz1} have developed a framework to obtain error bounds for the Kolmogorov and Wasserstein distance metrics for exponential approximation, and \cite{prr13} developed a framework for total variation error bounds for geometric approximation.  

This paper and its companion \cite{gaunt vgii} focuses on the development of Stein's method for VG approximation.  At the heart of the method \cite{gaunt vg} is the Stein equation
\begin{equation}\label{377}\sigma^2(x-\mu)f''(x)+(\sigma^2r+2\theta (x-\mu))f'(x)+(r\theta-(x-\mu))f(x)=\tilde{h}(x),
\end{equation}
where $\tilde{h}(x)=h(x)-\mathbb{E}h(Z)$ for $h:\mathbb{R}\rightarrow\mathbb{R}$ and $Z\sim\mathrm{VG}(r,\theta,\sigma,\mu)$.  Together with the Stein equations of \cite{pekoz} and \cite{pike}, this was one of the first second order Stein equations to appear in the literature.  We now set $\mu=0$; the general case follows from the translation property that if $Z\sim\mathrm{VG}(r,\theta,\sigma,\mu)$ then $Z-\mu\sim\mathrm{VG}(r,\theta,\sigma,0)$.  The solution to (\ref{377}) is then
\begin{align}  f_h(x) &=-\frac{\mathrm{e}^{-\theta x/\sigma^2}}{\sigma^2|x|^{\nu}}K_{\nu}\bigg(\!\frac{\sqrt{\theta^2+\sigma^2}}{\sigma^2}|x|\!\bigg)\! \int_0^x \! \mathrm{e}^{\theta t/\sigma^2} |t|^{\nu} I_{\nu}\bigg(\!\frac{\sqrt{\theta^2+\sigma^2}}{\sigma^2}|t|\!\bigg) \tilde{h}(t) \,\mathrm{d}t \nonumber \\
\label{vgsolngeneral0} &\quad-\frac{\mathrm{e}^{-\theta x/\sigma^2}}{\sigma^2|x|^{\nu}}I_{\nu}\bigg(\!\frac{\sqrt{\theta^2+\sigma^2}}{\sigma^2}|x|\!\bigg)\! \int_x^{\infty}\! \mathrm{e}^{\theta t/\sigma^2} |t|^{\nu} K_{\nu}\bigg(\!\frac{\sqrt{\theta^2+\sigma^2}}{\sigma^2}|t|\!\bigg)\tilde{h}(t)\,\mathrm{d}t,
\end{align}
where $\nu=\frac{r-1}{2}$ and the modified Bessel function of the first kind $I_\nu(x)$ is defined in Appendix \ref{appendix}.  If $h$ is bounded, then $f_h(x)$ and $f_h'(x)$ are bounded for all $x\in\mathbb{R}$.  Moreover, this is the unique bounded solution when $r\geq1$.

To approximate a random variable of interest $W$ by a VG random variable $Z$, one may evaluate both sides of (\ref{377}) at $W$, take expectations and finally take the supremum of both sides over a class of functions $\mathcal{H}$ to obtain
\begin{equation}\label{177}\sup_{h\in\mathcal{H}}|\mathbb{E}h(W)-\mathbb{E}h(Z)|
=\sup_{h\in\mathcal{H}}\big|\mathbb{E}\big[\sigma^2Wf_h''(W)+(\sigma^2r+2\theta W)f_h'(W)-(r\theta-W)f_h(W)\big]\big|.
\end{equation}
Many important probability metrics are of the form $\sup_{h\in\mathcal{H}}|\mathbb{E}h(W)-\mathbb{E}h(Z)|$.  In particular, taking
\begin{align*}\mathcal{H}_{\mathrm{K}}&=\{\mathbf{1}(\cdot\leq z)\,|\,z\in\mathbb{R}\}, \\
\mathcal{H}_{\mathrm{W}}&=\{h:\mathbb{R}\rightarrow\mathbb{R}\,|\,\text{$h$ is Lipschitz, $\|h'\|\leq1$}\}, \\
\mathcal{H}_{\mathrm{BW}}&=\{h:\mathbb{R}\rightarrow\mathbb{R}\,|\,\text{$h$ is Lipschitz, $\|h\|\leq1$ and $\|h'\|\leq1$}\}
\end{align*}
gives the Kolmogorov, Wasserstein and bounded Wasserstein distances, which we denote by $d_{\mathrm{K}}$, $d_{\mathrm{W}}$ and $d_{\mathrm{BW}}$, respectively.

The problem of bounding $\sup_{h\in\mathcal{H}}|\mathbb{E}h(W)-\mathbb{E}h(Z)|$ is thus reduced to bounding the solution (\ref{vgsolngeneral0}) and some of its lower order derivatives and bounding the expectation on the right-hand side of (\ref{177}).  To date, the only techniques for bounding this expectation for VG approximation are local couplings \cite{gaunt thesis, gaunt vg} and the integration by parts technique used to prove Theorem 4.1 of \cite{eichelsbacher}.  Other coupling techniques that are commonly found in the Stein's method literature, such as exchangeable pairs \cite{stein2} and Stein couplings \cite{cr10}, have yet to be used in VG approximation, although one of the contributions of this paper is a new coupling technique for SVG approximation by Stein's method.

The presence of modified Bessel functions in the solution (\ref{vgsolngeneral0}) together with the singularity at the origin in the Stein equation (\ref{377}) makes bounding the solution and its derivatives technically challenging.  Indeed, in spite of the introduction of new inequalities for modified Bessel functions and their integrals \cite{gaunt ineq1, gaunt ineq2} and extensive calculations (\cite{gaunt thesis}, Section 3.3 and Appendix D), the first bounds given in the literature \cite{gaunt vg} were only given for the case $\theta=0$ and had a far from optimal dependence on the parameter $r$.  Substantial progress was made by \cite{dgv15}, in which their iterative approach reduced the problem of bounding the derivatives of any order to bounding just the solution and its first derivative.  However, the bounds obtained in \cite{dgv15} have a dependence on the test function $h$ which means that error bounds for VG approximation can only be given for smooth test functions.  

\subsection{Summary of results and outline of the paper}

In this paper and its companion \cite{gaunt vgii}, we obtain new bounds for the solution of the VG Stein equation that allow for Wasserstein and Kolmogorov error bounds for VG approximation via Stein's method.  This paper focuses on the case $\theta=0$ (symmetric variance-gamma (SVG) distributions), while \cite{gaunt vgii} deals with the whole family of VG distributions.  This organisation is due to the additional complexity of the $\theta\not=0$ case.  One of the difficulties is that when $\theta\not=0$, the inequalities for expressions involving integrals  of modified Bessel functions that we use to bound the solution take a more complicated form, meaning our main results need to be presented in parallel for the two cases.  It should be noted, though, that, once the inequalities for modified Bessel functions have been established (which has now been done in \cite{gaunt ineq1, gaunt ineq2, gaunt ineq3}), the intrinsic difficulty of bounding the derivatives of the solution of the Stein equation in the two cases is similar.  This organisation allows for a clear exposition with manageable calculations.


In Section \ref{sec3}, we obtain new bounds for the solution of the SVG Stein equation (Theorem \ref{thm1} and Corollary \ref{cor345}) that have the correct dependence on the test function $h$ to allow for Wasserstein ($\|h'\|$) and Kolmogorov ($\|\tilde{h}\|$) error bounds for SVG approximation via Stein's method.  This task is arguably more technically demanding than for any other distribution for which this ingredient of Stein's method has been established.  Indeed, Theorem \ref{thm1} builds on the bounds of \cite{gaunt thesis, gaunt vg}, the iterative technique of \cite{dgv15}, and three papers on inequalities for integrals of modified Bessel functions \cite{gaunt ineq1, gaunt ineq2, gaunt ineq3} whose primary motivation was Stein's method for VG approximation.  In Propositions \ref{disclem} and \ref{ptpt} we note that higher order derivatives of the solution cannot have a dependence on $h$ of the form $\|\tilde{h}\|$ or $\|h'\|$.

In Section \ref{sec4}, we introduce (Definition \ref{def123}) a distributional transformation, which we call the \emph{centered equilibrium transformation of order $r$}, that is natural in the context of SVG approximation via Stein's method.  As our choice of name suggests, it generalises the centered equilibrium transformation \cite{pike}, which is itself the natural analogue for Laplace approximation of the equilibrium transformation for exponential approximation \cite{pekoz1}.    In Theorem \ref{jazzz}, we combine with the bounds of Section \ref{sec3} to obtain general Wasserstein and Kolmogorov error bounds for SVG approximation.  Our bounds are the SVG analogue of the general bounds of Theorem 3.1 of \cite{pekoz1} that have been shown to be a useful tool for obtaining bounds for exponential approximation.

It should be noted that even with the new bounds of Section \ref{sec3}, with other coupling techniques, such as local couplings, more effort may be required to obtain Wasserstein and Kolmogorov bounds than would be the case for normal approximation, for example.  This is due to the presence of the coefficient $\sigma^2x$ in the leading derivative of the SVG Stein equation (\ref{377}).  This therefore provides motivation for introducing this distributional transformation.

In Section \ref{sec5}, we apply the results of Sections \ref{sec3} and \ref{sec4} in four applications, these being: approximation of a general VG distribution by a SVG distribution; quantitative six moment theorems for SVG approximation of double Wiener-It\^{o} integrals; SVG approximation of a statistic for binary sequence comparison (a special case of the $D_2$ statistic for alignment-free sequence comparison \cite{bla, lippert}); and Laplace approximation of a random sum of independent mean zero random variables.  Our error bounds are given in the Wasserstein and Kolmogorov metrics, and in each case such bounds would not have been attainable by appealing to the present literature.

The rest of this paper is organised as follows.  In Section \ref{sec2}, we introduce the class of SVG distributions and state some of their basic properties.  Section \ref{sec3} gives new bounds for the solution of the SVG Stein equation.  In Section \ref{sec4}, we introduce a new distributional transformation, which we apply to give general bounds for SVG approximation in the Wasserstein and Kolmogorov metrics.  In Section \ref{sec5}, we apply our results to obtain SVG approximations in several applications.  Proofs of technical results are postponed until Section \ref{sec6}.  Basic properties and inequalities for modified Bessel functions that are needed in this paper are collected in Appendix \ref{appendix}.

\section{The class of symmetric variance-gamma distributions}\label{sec2}

In this section, we introduce the class of symmetric variance-gamma (SVG) distributions and present some of their basic properties.  

\begin{definition}If $Z\sim \mathrm{VG}(r,0,\sigma,\mu)$, for $r$, $\sigma$ and $\mu$ defined as in (\ref{vgdefn}), then $Z$ is said to have a \emph{symmetric variance-gamma} distribution.  We write $Z\sim \mathrm{SVG}(r,\sigma,\mu)$.
\end{definition}

Setting $\theta=0$ in (\ref{vgdefn}) gives the p.d.f$.$ of $Z\sim \mathrm{SVG}(r,\sigma,\mu)$:
\begin{equation}\label{svgdefn}p(x) = \frac{1}{\sigma\sqrt{\pi} \Gamma(\frac{r}{2})}  \bigg(\frac{|x-\mu|}{2\sigma}\bigg)^{\frac{r-1}{2}} K_{\frac{r-1}{2}}\bigg(\frac{|x-\mu|}{\sigma} \bigg), \quad x\in\mathbb{R},
\end{equation}
where $K_\nu(x)$ is a modified Bessel function of the second kind. The parameter $r$ is known as the \emph{scale} parameter. As $r$ increases, the distribution becomes more rounded around its peak value $\mu$ (as can be seen from (\ref{pmutend}) below). The parameter $\sigma$ is called the \emph{tail} parameter. As $\sigma$ decreases, the tails decay more quickly (see (\ref{pjv})).  The parameter $\mu$ is the \emph{location} parameter. Calculations can often be simplified by using the basic relation that if $Z\sim \mathrm{SVG}(r,1,0)$, then $\sigma Z+\mu \sim \mathrm{SVG}(r,\sigma,\mu)$.  The $\mathrm{SVG}(r,1,0)$ distribution is in a sense the \emph{standard} symmetric variance-gamma distribution.

The presence of the modified Bessel function makes (\ref{svgdefn}) difficult to parse at first inspection.  The following asymptotic formulas help in this regard.  Applying (\ref{Ktendinfinity}) to (\ref{svgdefn}) gives that, for all $r>0$, $\sigma>0$ and $\mu\in\mathbb{R}$,
\begin{equation}\label{pjv}p(x)\sim \frac{1}{(2\sigma)^{\frac{r}{2}}\Gamma(\frac{r}{2})}|x|^{\frac{r}{2}-1}\mathrm{e}^{-|x-\mu|/\sigma}, \quad |x|\rightarrow\infty.
\end{equation}
Similarly, applying (\ref{Ktend0}) to (\ref{svgdefn}) (see \cite{gaunt thesis}) gives that
\begin{equation}\label{pmutend}p(x)\sim\begin{cases}\displaystyle \frac{1}{\sigma\sqrt{\pi}}\frac{\Gamma\big(\frac{r-1}{2}\big)}{\Gamma\big(\frac{r}{2}\big)}, &  x\rightarrow\mu,\:r>1, \\
\displaystyle -\frac{1}{\pi\sigma}\log|x-\mu|, & x\rightarrow\mu,\: r=1, \\
\displaystyle \frac{1}{(2\sigma)^r\sqrt{\pi}}\frac{\Gamma\big(\frac{1-r}{2}\big)}{\Gamma\big(\frac{r}{2}\big)}|x-\mu|^{r-1}, & x\rightarrow\mu,\: 0<r<1. \end{cases}
\end{equation}
The density thus has a singularity at $x=\mu$ if $r\leq1$.  In fact, for any parameters, the $\mathrm{SVG}(r,\sigma,\mu)$ distribution is unimodal with mode $\mu$.  This can be seen from the fact that the function $x^\nu K_\nu(x)$ is a decreasing function of $x$ in the interval $(0,\infty)$ for all $\nu\in\mathbb{R}$ (see (\ref{ddbk})).

The SVG distribution has a fundamental representation in terms of independent normal and gamma random variables (\cite{kkp01}, Proposition 4.1.2).  Let $X\sim\Gamma(\frac{r}{2},\frac{1}{2})$ (with p.d.f$.$ $\frac{1}{2^{r/2}\Gamma(r/2)}x^{r/2-1}\mathrm{e}^{-x/2}$, $x>0$) and $Y\sim N(0,1)$ be independent.  Then $\mu+\sigma \sqrt{X}Y\sim\mathrm{SVG}(r,\sigma,\mu)$.

The SVG distribution has moment generating function $M(t)=(1+\sigma^2t^2)^{-r/2}$, $t\in\mathbb{R}$, and therefore has moments of arbitrary order.  In particular, the mean and variance of $Z\sim \mathrm{SVG}(r,\sigma,\mu)$ are given by
\begin{equation} \label{vgmoment}\mathbb{E}Z=\mu, \quad \mathrm{Var}(Z)=r\sigma^2.
\end{equation}
Perhaps surprisingly, this author could not find a formula for the absolute centered moments of the $\mathrm{SVG}(r,\sigma,\mu)$ distribution in the literature.  The result and its simple proof are given here.



\begin{proposition}Let $Z\sim\mathrm{SVG}(r,\sigma,\mu)$.  Then, for $k>0$,
\begin{equation}\label{vgabmom}\mathbb{E}|Z-\mu|^k=\frac{2^{\frac{k}{2}}\sigma^k}{\sqrt{\pi}}\frac{\Gamma\big(\frac{r+k}{2}\big)\Gamma\big(\frac{k+1}{2}\big)}{\Gamma\big(\frac{r}{2}\big)}.
\end{equation}
\end{proposition}

\begin{proof}We follow the approach given in Proposition 4.1.6 of \cite{kkp01} to obtain the moments of the $\mathrm{SVG}(r,\sigma,0)$ distribution. Recall that $Z-\mu=_d\sigma \sqrt{X} Y$, where $X\sim\Gamma(\frac{r}{2},\frac{1}{2})$ and $Y\sim N(0,1)$ are independent.  Therefore
\[\mathbb{E}|Z-\mu|^k=\sigma^k\mathbb{E}[X^{\frac{k}{2}}]\mathbb{E}|Y|^{k},\]
whence the result follows on using the standard formulas $\mathbb{E}X^s=\Gamma(\frac{r}{2}+s)/\Gamma(\frac{r}{2})$ and $\mathbb{E}|Y|^k=\frac{1}{\sqrt{\pi}}2^{\frac{k}{2}}\Gamma\big(\frac{k+1}{2}\big)$.
\end{proof}
In interpreting Corollary \ref{cor5.4} it will be useful to note the following formulas for the moments and cumulants of $Y\sim \mathrm{SVG}(r,\sigma,0)$ (\cite{eichelsbacher}, Lemma 3.6):
\begin{align*}&\mathbb{E}Y^2=r\sigma^2, \quad \mathbb{E}Y^4=3\sigma^4r(r+2), \quad \mathbb{E}Y^6=15\sigma^6r(r+2)(r+4), \\
&\kappa_2(Y)=r\sigma^2, \quad \kappa_4(Y)=6r\sigma^4,\quad \kappa_6(Y)=120r\sigma^6,
\end{align*}
with the odd order moments and cumulants all being equal to zero.  

Lastly, we note that the class of SVG distributions contains several classical distributions as special or limiting cases (\cite{gaunt vg}, Proposition 1.2).  

\begin{enumerate}

\item Let $X_r$ have the $\mathrm{SVG}(r,\frac{\sigma}{\sqrt{r}},\mu)$ distribution.  Then $X_r$ converges in distribution to a $N(\mu,\sigma^2)$ random variable in the limit $r\rightarrow\infty$.

\item A $\mathrm{SVG}(2,\sigma,\mu)$ random variable has the $\mathrm{Laplace}(\mu,\sigma)$ distribution with p.d.f$.$ $p(x)=\frac{1}{2\sigma}\mathrm{e}^{-|x-\mu|/\sigma}$, $x\in\mathbb{R}$. 

\item Let $X_1,\ldots,X_r$ and $Y_1,\ldots,Y_r$ be independent standard normal random variables.  Then $\sigma\sum_{k=1}^rX_kY_k$ has the $\mathrm{SVG}(r,\sigma,0)$ distribution.  

\item Suppose that $X\sim \Gamma(r,\lambda)$ and $Y\sim \Gamma(r,\lambda)$ are independent.  Then the random variable $X-Y$ has the $\mathrm{SVG}(2r,\lambda^{-1},0)$ distribution.

\end{enumerate}

\section{Bounds for the solution of the Stein equation}\label{sec3}

In this section, we obtain bounds for the solution of the SVG Stein equation (that is (\ref{377}) with $\theta=0$) which have the correct dependence on the test function $h$ to allow for Wasserstein and Kolmgorov distance bounds for SVG approximation via Stein's method.  

For ease of exposition, in our proofs, we shall analyse the solution of the $\mathrm{SVG}(r,1,0)$ Stein equation.  The general case follows from that fact that $\mathrm{SVG}(r,\sigma,\mu)=_d\mu+\sigma\mathrm{SVG}(r,1,0)$ and a simple rescaling and translation.  The solution of the $\mathrm{SVG}(r,1,0)$ Stein equation is then
\begin{align}\label{vgsolngeneral01}   f(x) &=-\frac{K_{\nu}(|x|)}{|x|^{\nu}}\! \int_0^x \! |t|^{\nu} I_{\nu}(|t|) \tilde{h}(t) \,\mathrm{d}t -\frac{I_{\nu}(|x|)}{|x|^{\nu}}\! \int_x^{\infty}\! |t|^{\nu} K_{\nu}(|t|)\tilde{h}(t)\,\mathrm{d}t \\
\label{vgsolngeneral11}&=-\frac{K_{\nu}(|x|)}{|x|^{\nu}} \!\int_0^x \! |t|^{\nu} I_{\nu}(|t|) \tilde{h}(t) \,\mathrm{d}t +\frac{I_{\nu}(|x|)}{|x|^{\nu}}\! \int_{-\infty}^{x}\! |t|^{\nu} K_{\nu}(|t|)\tilde{h}(t)\,\mathrm{d}t,
\end{align}
where $\nu=\frac{r-1}{2}$ and $\tilde{h}(x)=h(x)-\mathbb{E}h(Z)$ for $Z\sim\mathrm{SVG}(r,1,0)$.  The equality between (\ref{vgsolngeneral01}) and (\ref{vgsolngeneral11}) follows because $|t|^{\nu} K_{\nu}(|t|)$ is proportional to the $\mathrm{SVG}(r,1,0)$ density.  The equality is very useful, because it means that we will be able to restrict our attention to bounding the solution in the region $x\geq0$, from which a bound for all $x\in\mathbb{R}$ is immediate.

We now note two useful bounds due to \cite{gaunt vg} for the solution of the $\mathrm{SVG}(r,\sigma,\mu)$ Stein equation that will be used in the proof of Theorem \ref{thm1} and some of the applications of Section \ref{sec5}.  For bounded and measurable $h:\mathbb{R}\rightarrow\mathbb{R}$,
\begin{eqnarray}\label{vgf0}\|f\|&\leq& \frac{1}{\sigma}\bigg(\frac{1}{r}+\frac{\pi\Gamma(\frac{r}{2})}{2\Gamma(\frac{r+1}{2})}\bigg)\|\tilde{h}\|,  \\
\label{vgf1}\|f'\| &\leq& \frac{2}{\sigma^2r}\|\tilde{h}\|.
\end{eqnarray}

Let us now state the main result of this section.

\begin{theorem}\label{thm1}Suppose that $h:\mathbb{R}\rightarrow\mathbb{R}$ is bounded and measurable.  Let $f$ be the solution of the $\mathrm{SVG}(r,\sigma,\mu)$ Stein equation.  Then
\begin{eqnarray}\label{thmone}\|(x-\mu)f(x)\|&\leq& \bigg(\frac{3}{2}+\frac{1}{2r}\bigg)\|\tilde{h}\|, \\
\label{thmtwo}\|(x-\mu)f'(x)\|&\leq& \frac{1}{\sigma}\bigg(1+\frac{1}{2r}\bigg)\|\tilde{h}\|, \\
\label{thmthree}\|(x-\mu)f''(x)\|&\leq& \frac{1}{2\sigma^2}\bigg(9+\frac{1}{r}\bigg)\|\tilde{h}\|.
\end{eqnarray}
Now suppose that $h:\mathbb{R}\rightarrow\mathbb{R}$ is Lipschitz.  Then 
\begin{eqnarray}\label{thm1f0}\|f\|&\leq&\frac{7}{2}\|h'\|, \\
\|f'\|\label{thm1f1}&\leq&\frac{9}{2\sigma}\bigg(\frac{1}{r+1}+\frac{\pi\Gamma\big(\frac{r+1}{2}\big)}{2\Gamma\big(\frac{r}{2}+1\big)}\bigg)\|h'\|, \\
\|f''\|\label{thm1f2}&\leq&\frac{9}{\sigma^2(r+1)}\|h'\|,
\end{eqnarray}
and also
\begin{eqnarray}
\label{thmfour}\|(x-\mu)f'(x)\|&\leq&\frac{9}{2}\bigg(\frac{3}{2}+\frac{1}{2(r+1)}\bigg)\|h'\|, \\
\label{thmfive}\|(x-\mu)f''(x)\|&\leq&\frac{9}{2\sigma}\bigg(1+\frac{1}{2(r+1)}\bigg)\|h'\|, \\
\label{thmsix}\|(x-\mu)f^{(3)}(x)\|&\leq&\frac{9}{4\sigma^2}\bigg(9+\frac{1}{r+1}\bigg)\|h'\|.
\end{eqnarray}
\end{theorem}

\begin{proof}As noted above, for ease of notation, we set $\sigma=1$ and $\mu=0$.  The bounds for the general case, as stated in the theorem, follow from a simple change of variables; see the proof of Theorem 3.6 of \cite{gaunt vg}.  We also recall that it suffices to obtain bounds in the region $x\geq0$.

Let us first establish the bound for $\|f\|$, which we will need to obtain several of the other bounds.  By the mean value theorem, $|\tilde{h}(x)|\leq\|h'\|(|x|+\mathbb{E}|Z|)$, where $Z\sim \mathrm{SVG}(r,1,0)$.  From (\ref{vgabmom}) we have $\mathbb{E}|Z|=\frac{2}{\sqrt{\pi}}\Gamma(\frac{r+1}{2})/\Gamma(\frac{r}{2})$.  Now, on using inequalities (\ref{propb2a12}), (\ref{rnmt2}), (\ref{fff1}) and (\ref{rnmt1}) to obtain the second inequality we have, for $x\geq0$,
\begin{align*}|f(x)|&\leq\|h'\|\bigg\{\frac{K_\nu(x)}{x^\nu}\!\int_0^x\!(t+\mathbb{E}|Z|)t^\nu I_\nu(t)\,\mathrm{d}t+\frac{I_\nu(x)}{x^\nu}\!\int_x^\infty\!(t+\mathbb{E}|Z|)t^\nu K_\nu(t)\,\mathrm{d}t\bigg\}
\\ 
&\leq \|h'\|\bigg\{\frac{1}{2}+\mathbb{E}|Z|\cdot\frac{1}{2\nu+1}+1+\mathbb{E}|Z|\cdot\frac{\sqrt{\pi}\Gamma(\nu+\frac{1}{2})}{2\Gamma(\nu+1)}\bigg\} \\
&=\|h'\|\bigg\{\frac{3}{2}+\frac{2}{\sqrt{\pi}}\frac{\Gamma(\nu+1)}{\Gamma(\nu+\frac{1}{2})}\bigg(\frac{1}{2\nu+1}+\frac{\sqrt{\pi}\Gamma(\nu+\frac{1}{2})}{2\Gamma(\nu+1)}\bigg)\bigg\} \\
&=\|h'\|\bigg\{\frac{5}{2}+\frac{\Gamma(\nu+1)}{\sqrt{\pi}\Gamma(\nu+\frac{3}{2})}\bigg\},
\end{align*}
where we used the standard formula $u\Gamma(u)=\Gamma(u+1)$ to obtain the final equality.  Now, $\frac{\Gamma(\nu+1)}{\Gamma(\nu+\frac{3}{2})}$ is a decreasing function of $\nu$ for $\nu>-\frac{1}{2}$ (see \cite{gio}), and so is bounded above by $\Gamma(\frac{1}{2})=\sqrt{\pi}$ for all $\nu>-\frac{1}{2}$.  Hence, $|f(x)|\leq\frac{7}{2}\|h'\|$ for all $x\geq0$, which is sufficient to prove (\ref{thm1f0}).

The bounds for $\|f'\|$ and $\|f''\|$ can be obtained through an application of the iterative technique of \cite{dgv15}.  Differentiating both sides of the $\mathrm{SVG}(r,1,0)$ Stein equation (\ref{377}) gives
\begin{equation}\label{iteqn}xf^{(3)}(x)+(r+1)f''(x)-xf'(x)=h'(x)+f(x).
\end{equation}
We recognise (\ref{iteqn}) as the $\mathrm{SVG}(r+1,1,0)$ Stein equation, applied to $f'$, with test function $h'(x)+f(x)$.  We note that the test function $h'(x)+f(x)$ has mean zero with respect to the random variable $Y\sim\mathrm{SVG}(r+1,1,0)$.  This fact will be required in order to later apply inequalities (\ref{vgf0}) and (\ref{vgf1}).  Since $h$ is Lipschitz, by inequality (\ref{thm1f0}), we have that $\mathbb{E}|h'(Y)+f(Y)|<\infty$, and in particular as (\ref{iteqn}) is the $\mathrm{SVG}(r+1,1,0)$ Stein equation applied to $f'$, we have that $\mathbb{E}[Yf^{(3)}(Y)+(r+1)f''(Y)-Yf'(Y)]=0$, and thus $\mathbb{E}[h'(Y)+f(Y)]=0$.    
Therefore applying inequalities (\ref{vgf0}) and (\ref{vgf1}), respectively, with $r$ replaced by $r+1$ and test function $h'(x)+f(x)$, gives
\begin{align*}\|f'\|&=\bigg(\frac{1}{r+1}+\frac{\pi\Gamma\big(\frac{r+1}{2}\big)}{2\Gamma\big(\frac{r}{2}+1\big)}\bigg)\|h'(x)+f(x)\| \\
&\leq\bigg(\frac{1}{r+1}+\frac{\pi\Gamma\big(\frac{r+1}{2}\big)}{2\Gamma\big(\frac{r}{2}+1\big)}\bigg)\big(\|h'\|+\|f\|\big)\leq \frac{9}{2}\bigg(\frac{1}{r+1}+\frac{\pi\Gamma\big(\frac{r+1}{2}\big)}{2\Gamma\big(\frac{r}{2}+1\big)}\bigg)\|h'\|, \\
\|f''\|&\leq \frac{2}{r+1}\big(\|h'\|+\|f\|\big)\leq \frac{9}{r+1}\|h'\|,
\end{align*}
where we used (\ref{thm1f0}) to bound $\|f\|$.

Let us now establish the bounds (\ref{thmone}), (\ref{thmtwo}) and (\ref{thmthree}).  On using inequalities (\ref{jjj1}) and (\ref{ddd2}) to obtain the second inequality, we obtain, for $x\geq0$,
\begin{align*}|xf(x)|&=\bigg|\frac{K_\nu(x)}{x^{\nu-1}}\int_0^x t^\nu I_\nu(t)\tilde{h}(t)\,\mathrm{d}t-\frac{I_\nu(x)}{x^{\nu-1}}\int_x^\infty t^\nu K_\nu(t)\tilde{h}(t)\,\mathrm{d}t\bigg| \\
&\leq\|\tilde{h}\|\bigg\{\frac{K_\nu(x)}{x^{\nu-1}}\int_0^x t^\nu I_\nu(t)\,\mathrm{d}t+\frac{I_\nu(x)}{x^{\nu-1}}\int_x^\infty t^\nu K_\nu(t)\,\mathrm{d}t\bigg\} \\
&\leq \|\tilde{h}\|\bigg(\frac{\nu+1}{2\nu+1}+1\bigg)=\frac{3\nu+2}{2\nu+1}\|\tilde{h}\|=\frac{3r+1}{2r}\|\tilde{h}\|=\bigg(\frac{3}{2}+\frac{1}{2r}\bigg)\|\tilde{h}\|.
\end{align*}
On using the differentiation formulas (\ref{diff11}) and (\ref{diff22}) and inequalities (\ref{jjj1z}) and (\ref{ddd2z}), we obtain, for $x\geq0$,
\begin{align*}|xf'(x)|&=\bigg|\frac{K_{\nu+1}(x)}{x^{\nu-1}}\int_0^x t^\nu I_\nu(t)\tilde{h}(t)\,\mathrm{d}t+\frac{I_{\nu+1}(x)}{x^{\nu-1}}\int_x^\infty t^\nu K_\nu(t)\tilde{h}(t)\,\mathrm{d}t\bigg| \\
&\leq \|\tilde{h}\|\bigg\{\frac{K_{\nu+1}(x)}{x^{\nu-1}}\int_0^x t^\nu I_\nu(t)\,\mathrm{d}t+\frac{I_{\nu+1}(x)}{x^{\nu-1}}\int_x^\infty t^\nu K_\nu(t)\,\mathrm{d}t\bigg\} \\
&\leq \|\tilde{h}\|\bigg(\frac{\nu+1}{2\nu+1}+\frac{1}{2}\bigg)=\frac{4\nu+3}{2(2\nu+1)}\|\tilde{h}\|=\frac{2r+1}{2r}\|\tilde{h}\|=\bigg(1+\frac{1}{2r}\bigg)\|\tilde{h}\|.
\end{align*}
Since it suffices to consider $x\geq0$, we have proved (\ref{thmone}) and (\ref{thmtwo}).  Now, from the $\mathrm{SVG}(r,1,0)$ Stein equation we have that, for $x\in\mathbb{R}$, 
\begin{equation*}|xf''(x)|=|\tilde{h}(x)-rf'(x)+xf(x)|\leq \|\tilde{h}\|+r\|f'\|+\|xf(x)\|.
\end{equation*}
Applying (\ref{vgf1}) to bound $\|f'\|$ and (\ref{thmone}) to bound $\|xf(x)\|$ yields the bound
\begin{equation*}\|xf''(x)\|\leq \bigg(1+r\cdot\frac{2}{r}+\bigg(\frac{3}{2}+\frac{1}{2r}\bigg)\bigg)\|\tilde{h}\|=\frac{1}{2}\bigg(9+\frac{1}{r}\bigg)\|\tilde{h}\|.
\end{equation*}


We now bound $\|xf'(x)\|$, $\|xf''(x)\|$ and $\|xf^{(3)}(x)\|$ for Lipschitz $h$ using the iterative technique of \cite{dgv15} similarly to how we obtained inequalities (\ref{thm1f1}) and (\ref{thm1f2}).  We recall that (\ref{iteqn}) is the $\mathrm{SVG}(r+1,1,0)$ Stein equation, applied to $f'$, with test function $h'(x)+f(x)$, which we also recall has mean zero with respect to the $\mathrm{SVG}(r+1,1,0)$ measure.   Therefore applying inequalities (\ref{thmone}), (\ref{thmtwo}) and (\ref{thmthree}), respectively, with $r$ replaced by $r+1$ and test function $h'(x)+f(x)$, gives
\begin{eqnarray*}\|xf'(x)\|&\leq&\bigg(\frac{3}{2}+\frac{1}{2(r+1)}\bigg)\|h'(x)+f(x)\|\\
&\leq&\bigg(\frac{3}{2}+\frac{1}{2(r+1)}\bigg)\big(\|h'\|+\|f\|\big)\leq \frac{9}{2}\bigg(\frac{3}{2}+\frac{1}{2(r+1)}\bigg)\|h'\|, \\
\|xf''(x)\|&\leq&\bigg(1+\frac{1}{2(r+1)}\bigg)\big(\|h'\|+\|f\|\big)\leq \frac{9}{2}\bigg(1+\frac{1}{2(r+1)}\bigg)\|h'\|, \\
\|xf^{(3)}(x)\|&\leq&\frac{1}{2}\bigg(9+\frac{1}{r+1}\bigg)\big(\|h'\|+\|f\|\big)\leq \frac{9}{4}\bigg(9+\frac{1}{r+1}\bigg)\|h'\|, 
\end{eqnarray*}
where we used inequality (\ref{thm1f0}) to bound $\|f\|$.  The proof is complete.
\end{proof}

In the following corollary, we apply some of the estimates of Theorem \ref{thm1} to bound some useful quantities.  We shall make use of these bounds in Section \ref{sec4}.  It will be convenient to define the operator $T_r$ by $T_rf(x)=xf'(x)+rf(x)$.

\begin{corollary}\label{cor3.22}Let $f$ be the solution of the $\mathrm{SVG}(r,\sigma,0)$ Stein equation.  Then for $h:\mathbb{R}\rightarrow\mathbb{R}$ bounded and measurable, and Lipschitz, respectively,  
\begin{eqnarray}\label{medfs}\sigma^2\|T_rf'\|&\leq&\bigg(\frac{5}{2}+\frac{1}{2r}\bigg)\|\tilde{h}\|, \\
\label{lemse}\sigma^2\|(T_rf')'\|&\leq&\frac{9}{4}\bigg(5+\frac{1}{r+1}\bigg)\|h'\|.
\end{eqnarray}
\end{corollary}

\begin{proof}As $f$ satisfies $\sigma^2T_rf'(x)=\tilde{h}(x)+xf(x)$, we have $\sigma^2(T_rf')'(x)=h'(x)+xf'(x)+f(x)$.  From the triangle inequality and the estimates of Theorem \ref{thm1},
\begin{align*}\sigma^2\|T_rf'\|&\leq\|\tilde{h}\|+\|xf(x)\|\leq \bigg(\frac{5}{2}+\frac{1}{2r}\bigg)\|\tilde{h}\|, \\
\sigma^2\|(T_rf')'\|&\leq\|h'\|+\|xf'(x)\|+\|f\|\leq \bigg\{1+\frac{9}{2}\bigg(\frac{3}{2}+\frac{1}{2(r+1)}\bigg)+\frac{7}{2}\bigg\}\|h'\|,
\end{align*}
which proves the result.
\end{proof}

\begin{corollary}\label{cor345}Let $f_z$ denote the solution of the $\mathrm{SVG}(r,\sigma,\mu)$ Stein equation with test function $h_z(x)=\mathbf{1}(x\leq z)$.  Then
\begin{align*}&\|f_z\|\leq \frac{1}{\sigma}\bigg(\frac{1}{r}+\frac{\pi\Gamma(\frac{r}{2})}{2\Gamma(\frac{r+1}{2})}\bigg), \quad \|f_z'\| \leq \frac{2}{\sigma^2r}, \quad \sigma^2\|T_rf_z'\|\leq \frac{5}{2}+\frac{1}{2r}, \\
&\|(x-\mu)f_z(x)\|\leq \frac{3}{2}+\frac{1}{2r}, \quad \|(x-\mu)f_z'(x)\|\leq \frac{1}{\sigma}\bigg(1+\frac{1}{2r}\bigg), \\
&\|(x-\mu)f_z''(x)\|\leq \frac{1}{2\sigma^2}\bigg(9+\frac{1}{r}\bigg).
\end{align*}
\end{corollary}

\begin{proof}Apply the inequality $\|\tilde{h}_z\|\leq 1$ to the bounds (\ref{vgf0}), (\ref{vgf1}), (\ref{medfs}), (\ref{thmone}), (\ref{thmtwo}) and (\ref{thmthree}), respectively.
\end{proof}

\begin{remark}For the normal \cite{chen} and exponential \cite{chatterjee} Stein equations, because the solution of the Stein equation with test function $h_z(x)=\mathbf{1}(x\leq z)$ can be expressed in terms of elementary functions, a detailed analysis of the solution yields bounds with smaller constants than would be obtained by first working with a general bounded test function $h$ and then bounding $\|\tilde{h}_z\|\leq1$.  However, because of the presence of modified Bessel functions in the solution, such improvements would be more difficult to obtain here.
\end{remark}

It is natural to ask whether, for all $z\in\mathbb{R}$, a bound of the form $\|f_z''\|\leq C_{r,\sigma}$ could be obtained for the solution $f_z$.  The following proposition, which is proved in Section \ref{sec6}, shows that this is not possible.

\begin{proposition}\label{disclem}$f_\mu'(x)$ has a discontinuity at $x=\mu$.
\end{proposition}

Similarly, one may ask whether a bound of the form $\|f^{(3)}\|\leq C_{r,\sigma}\|h'\|$ could be obtained for all Lipschitz $h:\mathbb{R}\rightarrow\mathbb{R}$.  The following proposition, which is proved in Section \ref{sec6}, again shows this is not possible (see \cite{eden2} for similar results that apply to solutions of Stein equations for a wide class of distributions).  Our approach differs from that of Proposition \ref{disclem} in that we do not find a Lipschitz test function $h$ for which $f''$ has a discontinuity.  This would be more tedious to establish for $f''$ than for $f'$ and instead we consider a highly oscillating test function and perform an asymptotic analysis.    

\begin{proposition}\label{ptpt}Let $f$ be the solution of the $\mathrm{SVG}(r,\sigma,\mu)$ Stein equation. Then there does not exist a constant $M_{r,\sigma}>0$ such that $\|f^{(3)}\|\leq M_{r,\sigma}\|h'\|$ for all Lipschitz $h:\mathbb{R}\rightarrow\mathbb{R}$.
\end{proposition}

\begin{remark}\label{remv}(i) Throughout this remark, we set $\mu=0$. The bounds (\ref{vgf1}) and (\ref{thm1f2}) are of order $r^{-1}$ as $r\rightarrow\infty$.  This is indeed the optimal order, which can be seen by the following argument, which is similar to the one given in Remark 2.2 of \cite{gaunt chi square} to show that the rate in their bound for solution of the gamma Stein equation was optimal. 

Evaluating both sides of the $\mathrm{SVG}(r,\sigma,0)$ Stein equation at $x=0$ gives $f'(0)=\frac{1}{\sigma^2r}\tilde{h}(0)$.  Also, evaluating both sides of (\ref{iteqn}) (with general $\sigma$) at $x=0$ gives that $f''(0)=\frac{1}{\sigma^2(r+1)}\big(h'(0)+f(0)\big)$, from which we conclude that the $O(r^{-1})$ rate in (\ref{thm1f2}) is also optimal.

(ii) The bound (\ref{vgf0}) for $\|f\|$ is of order $r^{-\frac{1}{2}}$ as $r\rightarrow\infty$.  Indeed, for $r>1$,
\begin{equation}\label{oct8}\sqrt{\frac{2}{r}}<\frac{\Gamma(\frac{r}{2})}{\Gamma(\frac{r+1}{2})}<\sqrt{\frac{2}{r-\frac{1}{2}}},
\end{equation}
which follows from the inequalities $\frac{\Gamma(x+\frac{1}{2})}{\Gamma(x+1)}>(x+\frac{1}{2})^{-\frac{1}{2}}$ for $x>0$ (see \cite{gaut}) and $\frac{\Gamma(x+\frac{1}{2})}{\Gamma(x+1)}<(x+\frac{1}{4})^{-\frac{1}{2}}$ for $x>-\frac{1}{4}$ (see \cite{elezovic}).  The $O(r^{-\frac{1}{2}})$ rate is optimal, which can be seen as follows.  Take $h$ to be $h(x)=1$ if $x\geq0$ and $h(x)=-1$ if $x<0$, so that $\tilde{h}(x)=h(x)$.   
Then
\begin{align*}f(0+) &=-\lim_{x\downarrow0}\bigg\{\frac{1}{\sigma^2x^{\nu}}K_{\nu}\bigg(\frac{x}{\sigma}\bigg) \int_0^x  t^{\nu} I_{\nu}\bigg(\frac{t}{\sigma}\bigg) \,\mathrm{d}t\bigg\} -\lim_{x\downarrow0}\bigg\{\frac{1}{\sigma^2x^{\nu}}I_{\nu}\bigg(\frac{x}{\sigma}\bigg) \int_x^{\infty}  t^{\nu} K_{\nu}\bigg(\frac{t}{\sigma}\bigg)\,\mathrm{d}t\bigg\} \\
&=\frac{\sqrt{\pi}\Gamma(\nu+\frac{1}{2})}{2\sigma^2\Gamma(\nu+1)}=\frac{\sqrt{\pi}\Gamma(\frac{r}{2})}{2\sigma^2\Gamma(\frac{r+1}{2})}.
\end{align*}
Here we used that the first limit is equal to zero by the asymptotic formulas (\ref{Itend0}) and (\ref{Ktend0}).  We computed the second limit using the asymptotic formula (\ref{Itend0}) and that the integrand is proportional to the density of the $\mathrm{SVG}(2\nu+1,\sigma,0)$ distribution.  Therefore, by (\ref{oct8}), we conclude that the optimal rate is order $r^{-\frac{1}{2}}$ as $r\rightarrow\infty$.

(iii) Arguing as we did in part (i), we have that $f(0)=\sigma^2(r+1)f''(0)-h'(0)$, which for a general Lipschitz test function $h$ is $O(1)$ (see bound (\ref{thm1f2})), and so the bound (\ref{thm1f0}) is of optimal order.

(iv) In light of inequalities (\ref{thm1f0})--(\ref{thm1f2}) one might expect inequalities (\ref{thmtwo}) and (\ref{thmthree}) to be of lower than (\ref{thmone}) as $r\rightarrow\infty$.  However, this is not the case.  A calculation involving L'H\^{o}pital's rule (which is given in Section \ref{sec6}) shows that, for any bounded $h:\mathbb{R}\rightarrow\mathbb{R}$,
\begin{equation}\label{sec6def}\lim_{x\rightarrow\infty}xf(x)=-\tilde{h}(\infty), \quad \lim_{x\rightarrow-\infty}xf(x)=\tilde{h}(-\infty),
\end{equation}
and from the $\mathrm{SVG}(r,\sigma,0)$ Stein equation and inequality (\ref{thm1f1}) we obtain
\begin{equation*}\lim_{x\rightarrow-\infty}xf''(x)=\frac{1}{\sigma^2}\big[\tilde{h}(-\infty)+\lim_{x\rightarrow-\infty}xf(x)\big]=\frac{2}{\sigma^2}\tilde{h}(-\infty).
\end{equation*}
Thus, inequalities (\ref{thmone}) and (\ref{thmthree}) are of optimal order in $r$.  We expect this to also be the case for inequalities (\ref{thmfour})--(\ref{thmsix}), although verifying this would involve a more detailed analysis, which we omit for space reasons.
\end{remark}

\section{The centered equilibrium transformation of order $r$}\label{sec4}

In this section, we introduce a new distributional transformation and apply it to obtain general Wasserstein and Kolmogorov error bounds for SVG approximation.

We begin with the following proposition which relates the Kolmogorov and Wasserstein distances between a general distribution and a SVG distribution.  This proposition is of interest, because Wasserstein distance bounds are often easier to obtain than Kolmogorov distance bounds through Stein's method.  The proof is deferred until Section \ref{sec6}.

\begin{proposition}\label{prop1}Let $Z\sim\mathrm{SVG}(r,\sigma,\mu)$.  Then, for any random variable $W$:

(i) If $r>1$,
\begin{equation}\label{pronf1}d_{\mathrm{K}}(W,Z)\leq\sqrt{\frac{1}{\sigma\sqrt{\pi}}\frac{\Gamma\big(\frac{r-1}{2}\big)}{\Gamma\big(\frac{r}{2}\big)}d_{\mathrm{W}}(W,Z)}.
\end{equation}

(ii) Suppose that $\sigma^{-1} d_{\mathrm{W}}(W,Z)<0.676$.  Then, if $r=1$,
\begin{equation}\label{pronf2}d_{\mathrm{K}}(W,Z)\leq \bigg\{2+\log\bigg(\frac{2}{\sqrt{\pi}}\bigg)+\frac{1}{2}\log\bigg(\frac{\sigma}{d_{\mathrm{W}}(W,Z)}\bigg)\bigg\}\sqrt{\frac{d_{\mathrm{W}}(W,Z)}{\pi\sigma}}.
\end{equation}

(iii) If $0<r<1$,
\begin{equation}\label{pronf3}d_{\mathrm{K}}(W,Z)\leq 2\bigg(\frac{\Gamma\big(\frac{1-r}{2}\big)}{\sqrt{\pi}2^{r-1} \Gamma\big(\frac{r}{2}\big)}\bigg)^{\frac{1}{r+1}}\big(\sigma^{-1}d_{\mathrm{W}}(W,Z)\big)^{\frac{r}{r+1}}.
\end{equation}
\end{proposition}

\begin{remark} 

(i) If $\sigma^{-1} d_{\mathrm{W}}(W,Z)=0.676$, then the upper bound in part (ii) is equal to 1.075, and is therefore uninformative.

(ii) Recall that $N(\mu,\sigma^2)=_d\lim_{r\rightarrow\infty}\mathrm{SVG}(r,\frac{\sigma}{\sqrt{r}},\mu)$.  Therefore from (\ref{pronf1}) and the limit $\lim_{x\rightarrow\infty}\frac{\sqrt{x}\Gamma(x-\frac{1}{2})}{\Gamma(x)}=1$ (see (\ref{oct8})), we recover the inequality (with obvious abuse of notation)
\begin{equation*}d_{\mathrm{K}}(W,N(\mu,\sigma^2))\leq\Big(\frac{2}{\pi\sigma^2}\Big)^{\frac{1}{4}}\sqrt{d_{\mathrm{W}}(W,N(\mu,\sigma^2))},
\end{equation*}
which is a special case of part 2 of Proposition 1.2 of \cite{ross}.  It is known (see \cite{chen}, p$.$ 48) that this bound gives the optimal rate under some conditions, but in other applications the rate is suboptimal.  Proposition \ref{proppl} gives an application in which the inequalities (\ref{pronf1}), (\ref{pronf2}) and (\ref{pronf3}) are not of optimal rate in $\delta=d_{\mathrm{W}}(W,Z)$; see Remark \ref{remopl}.  

\end{remark}

As in Section \ref{sec3}, we define the operator $T_r$ by $T_rf(x)=xf'(x)+rf(x)$.  We also denote $D=\frac{\mathrm{d}}{\mathrm{d}x}$.  From now until the end of this section, we set $\mu=0$.

\begin{definition}\label{def123}Let $W$ be a random variable with mean zero and variance $0<r\sigma^2<\infty$.  We say that $W^{V_r}$ has the $W$-\emph{centered equilibrium distribution of order $r$} if 
\begin{equation}\label{defeqn}\mathbb{E}Wf(W)=\sigma^2\mathbb{E}T_rf'(W^{V_r})
\end{equation}
for all twice-differentiable $f:\mathbb{R}\rightarrow\mathbb{R}$ such that the expectations in (\ref{defeqn}) exist.
\end{definition}

As we shall see later, it is convenient to write $\mathrm{Var}(W)=r\sigma^2$, because the variance of a $\mathrm{SVG}(r,\sigma,0)$ random variable is $r\sigma^2$. As the name suggests, the centered equilibrium distribution of order $r$ generalises the centered equilibrium distribution of $W$, denoted by $W^L$, that was introduced by \cite{pike}.  Its characterising equation is 
\begin{equation}\label{pikegh}\mathbb{E}f(W)-f(0)=\frac{1}{2}\mathbb{E}W^2\mathbb{E}f''(W^L).
\end{equation}
We also refer the reader to \cite{dobler} for a generalisation of (\ref{pikegh}) to all random variables $W$ with finite second moment.  The centered equilibrium distribution is itself the Laplace analogue of the equilibrium distribution that has been shown to be a useful tool in Stein's method for exponential approximation by \cite{pekoz1}.  We can see that $W^{V_2}=W^L$ by setting $f(x)=xg(x)$ in (\ref{pikegh}).  For $r\not=2$, a characterising equation of the form (\ref{pikegh}) is not useful.  To see this, recall that the Stein operator for the $\mathrm{SVG}(r,\sigma,0)$ distribution is $\mathcal{A}f(x)=\sigma^2xf''(x)+\sigma^2rf'(x)-xf(x)$.  Setting $f(x)=g(x)/x$ then gives
\begin{align*}\mathcal{A}g(x)=\sigma^2g''(x)+(r-2)\sigma^2\bigg(\frac{g'(x)}{x}-\frac{g(x)}{x^2}\bigg)-g(x),
\end{align*}
which has a singularity at $x=0$ if $r\not=2$.

We also note that $W^{V_1}=W^{*(2)}$, where $W^{*(2)}$ has the $W$-zero bias distribution of order 2 (see \cite{gaunt pn}).  This distributional transformation is a natural generalisation of the zero bias transformation (defined below) to the setting of Stein's method for products of independent standard normal random variables.  We shall make use of this fact in Section \ref{sec5.3}.


We now obtain an inverse of the operator $T_rD$.  This inverse operator will be used later in this section to establish properties of the centered equilibrium distribution of order $r$.  Recall that the $\mathrm{Beta}(r,1)$ distribution has p.d.f$.$ $p(x)=rx^{r-1}$, $0<x<1$.

\begin{lemma}\label{invlem}Let $B_r\sim \mathrm{Beta}(r,1)$ and $U\sim U(0,1)$ be independent, and define the operator $G_r$ by $G_rf(x)=\frac{x}{r}\mathbb{E}f(xUB_r)$.  Then, $G_r$ is the right-inverse of the operator $T_r D$ in the sense that
\begin{equation}\label{rean}T_rDG_rf(x)=f(x).
\end{equation}
Suppose now that $f$ is  twice-differentiable.  Then, for any $r\geq 1$,
\begin{equation}\label{leftinv}G_rT_r Df(x)=f(x)-f(0).
\end{equation}
Therefore, $G_r$ is the inverse of $T_r D$ when the domain of $T_r D$ is the space of all twice-differentiable functions $f$ on $\mathbb{R}$ with $f(0)=0$.
\end{lemma}

\begin{proof}We begin by obtaining a useful formula for $G_rf(x)=\frac{x}{r}\mathbb{E}f(xUB_r)$:
\begin{align}\label{g2gg}G_rf(x)&=\frac{x}{r}\int_0^1\!\int_0^1f(xub)rb^{r-1}\,\mathrm{d}b\,\mathrm{d}u =\int_0^x\!\int_0^t f(s)s^{r-1}t^{-r}\,\mathrm{d}s\,\mathrm{d}t.
\end{align}

We now use (\ref{g2gg}) to verify (\ref{rean}):
\begin{align*}T_rDG_r f(x)&=T_r\bigg(x^{-r}\int_0^xf(s)s^{r-1}\,\mathrm{d}s\bigg) \\
\quad&=x\bigg(-rx^{r-1}\int_0^x f(s)s^{r-1}\,\mathrm{d}s+x^{-r}\cdot f(x)x^{r-1}\bigg)+rx^{-r}\int_0^x f(s)s^{r-1}\,\mathrm{d}s\\
& =f(x).
\end{align*}

Finally, we verify relation (\ref{leftinv}).  We have
\begin{align*}G_rT_rDf(x)&=\int_0^x\!\int_0^t \big(sf''(s)+rf'(s)\big)s^{r-1}t^{-r}\,\mathrm{d}s\,\mathrm{d}t \\
&=\int_0^x\!t^{-r}\!\int_0^t \big(s^rf'(s)\big)'\,\mathrm{d}s\,\mathrm{d}t =\int_0^x f'(t)\,\mathrm{d}t =f(x)-f(0),
\end{align*}
as required.
\end{proof}

Before presenting some properties of the centered equilibrium distribution of order $r$, we recall two distributional transformations that are standard in the Stein's method literature.  If $W$ is a mean zero random variable with finite, non-zero variance $\sigma^2$, we say that $W^*$ has the $W$-zero biased distribution \cite{goldstein} if for all differentiable $f$ for which $\mathbb{E}Wf(W)$ exists,
\begin{equation*}\mathbb{E}Wf(W)=\sigma^2\mathbb{E}f'(W^*).
\end{equation*}
For any random variable $W$ with finite second moment, we say that $W^{\square}$ has the $W$-square bias distribution (\cite{chen}, pp$.$ 34--35) if for all $f$ such that $\mathbb{E}W^2f(W)$ exists,
\begin{equation*} \mathbb{E}W^2f(W)=\mathbb{E}W^2\mathbb{E}f(W^{\square}).
\end{equation*}
When $EW=0$, there is neat relationship between these distribution transformations: $W^*=_dU W^\square$, where $U\sim U(0,1)$ is independent of $W^\square$ (this is a slight variant of Proposition 2.3 \cite{chen}; see \cite{gaunt pn}, Proposition 3.2).

The following construction of $W^{V_r}$ generalises Theorem 3.2 of \cite{pike}.  Similar constructions for distributional transformations that are natural in the context in gamma and generalized gamma approximation can be found in \cite{pr12} and \cite{prr16}.

\begin{proposition}\label{propsquare}Let $W$ be a random variable with zero mean and finite, non-zero variance $r\sigma^2$, and let $W^{*}$ have the $W$-zero bias distribution.  Let $B_r\sim \mathrm{Beta}(r,1)$ be independent of $W^{*}$.   Then, the random variable
\begin{equation*}W^{V_r}=_dB_rW^{*}
\end{equation*}
has the centered equilibrium distribution of order $r$.
\end{proposition}

\begin{proof}Let $f\in C_c$, the collection of continuous functions with compact support.  In Lemma \ref{invlem} we defined the operator $G_rg(x)=\frac{x}{r}\mathbb{E}g(xUB_r)$ and showed that $T_rDG_rg(x)=g(x)$ for any $g$.  We therefore have
\begin{align*}\sigma^2\mathbb{E}f(W^{V_r})&=\sigma^2\mathbb{E}T_rDG_rf(W^{V_r})=\mathbb{E}WG_rf(W)=\frac{1}{r}\mathbb{E}W^2f(UB_rW)  \\
&=\frac{1}{r}\mathbb{E}W^2\mathbb{E}f(UB_r W^\square)=\sigma^2\mathbb{E}f(UB_r W^\square)=\sigma^2\mathbb{E}f(B_rW^{*}).
\end{align*}
Since the expectation of $f(W^{V_r})$ and $f(B_rW^*)$ are equal for all $f\in C_c$, the random variables $W^{V_r}$ and $B_rW^*$ must be equal in distribution.
\end{proof}

In the following proposition, we collect some useful properties of the centered equilibrium distribution of order $r$.  As might be expected in the light of Proposition \ref{propsquare}, some of these properties are quite similar to those given for the zero bias distribution in Lemma 2.1 of \cite{goldstein}.

\begin{proposition}\label{propsquare1}Let $W$ be a mean zero variable with finite, non-zero variance $r\sigma^2$, and let $W^{V_r}$ have the $W$-centered equilibrium distribution of order $r$ in accordance with Definition \ref{def123}.

(i) The $\mathrm{SVG}(r,\sigma,0)$ distribution is the unique fixed point of the centered equilibrium transformation of order $r$.

(ii) The distribution of $W^{V_r}$ is unimodal about zero and absolutely continuous with density
\begin{equation} \label{zxcv} f_{W^{V_r}}(w)=\frac{1}{\sigma^2}\int_0^1 t^{r-2}\mathbb{E}[W\mathbf{1}(W>w/t)]\,\mathrm{d}t.
\end{equation}
It follows that the support of $W^{V_r}$ is the closed convex hull of the support of $W$ and that $W^{V_r}$ is bounded whenever $W$ is bounded.

(iii) The centered equilibrium transformation of order $r$ preserves symmetry.

(iv) For $p\geq 0$,
\begin{equation*}\mathbb{E}[(W^{V_r})^p]=\frac{\mathbb{E}W^{p+2}}{r\sigma^2(p+1)(p+r)} \quad \mbox{and} \quad
\mathbb{E}|W^{V_r}|^p = \frac{\mathbb{E}|W|^{p+2}}{r\sigma^2(p+1)(p+r)}.
\end{equation*} 

(v) For $c\in\mathbb{R}$, $cW^{V_r}$ has the $cW$-centered equilibrium distribution of order $r$.
\end{proposition}

\begin{proof} (i) This is immediate from Definition \ref{def123} and the Stein characterisation for the $\mathrm{SVG}(r,\sigma,0)$ distribution given in Lemma 3.1 of \cite{gaunt vg}.

(ii)  Firstly, we note that, for fixed $t\in(0,1)$, the expectation $\mathbb{E}[W\mathbf{1}(W>w/t)]$ is increasing for $w<0$ and decreasing for $w>0$.  We therefore deduce that $p(w)$ is increasing for $w<0$ and decreasing for $w>0$.  Now, from Proposition \ref{propsquare}, we have that $W^{V_r}=_dB_r W^*$.  Formula (\ref{zxcv}) then follows from the fact that $X^*$ is absolutely continuous with density $f_{W^*}(w)=\mathbb{E}[W\mathbf{1}(W>w)]/\mathrm{Var}(W)$ (part (ii) of Lemma 2.1 of \cite{goldstein}) and the standard formula for computing the density of a product.

(iii) We follow the argument of part (iii) of Lemma 2.1 of \cite{goldstein}.  Let $w$ be a continuity point of a symmetric random variable $W$.  Then, for fixed $t\in(0,1)$, $\mathbb{E}[W\mathbf{1}(W>w/t)]=\mathbb{E}[-W\mathbf{1}(-W>w/t)]=-\mathbb{E}[W\mathbf{1}(W<-w/t)]=\mathbb{E}[W\mathbf{1}(W>-w/t)]$, using $\mathbb{E}W=0$.  It is now evident from (\ref{zxcv}) that $f_{W^{V_r}}(w)=f_{W^{V_r}}(-w)$ for almost all $w$.  Therefore, there is a version of the $\mathrm{d}w$ density of $W^{V_r}$ which is the same at $w$ and $-w$ for almost all $w[\mathrm{d}w]$, and so $W^{V_r}$ is symmetric.

(iv) Substitute $w^{p+1}$ and $|w|^{p+1}$ for $f(w)$ in the characterising equation (\ref{defeqn}).

(v) Let $g$ be a function such that $\mathbb{E}Wg(W)$ exists, and define $\tilde{g}(x)=cg(cx)$.  Then $\tilde{g}^{(k)}(x)=c^{k+1}g^{(k)}(cx)$.  As $W^{V_r}$ has the $W$-centered equilibrium distribution of order $r$,
\[\mathbb{E}cWg(cW)=\mathbb{E}W\tilde{g}(W)=\sigma^2\mathbb{E}T_rD\tilde{g}(W^{V_r})=(c\sigma)^2\mathbb{E}T_rDg(cW^{V_r}).\]
Hence $cW^{V_r}$ has the $cW$-centered equilibrium distribution of order $r$.
\end{proof}

We end this section by proving Theorem \ref{jazzz} below, which formalises the notion that if $\mathcal{L}(W)$ and $\mathcal{L}(W^{V_r})$ are approximately equal then $W$ has an approximation SVG distribution. This theorem is the SVG analogue of Theorem 2.1 of \cite{pekoz1}, in which the Wasserstein and Kolmogorov error bounds are given in terms of the difference in absolute expectation between the random variable of interest $W$ and its $W$-equilibrium transformation.  We follow the approach of \cite{pekoz1} and begin by stating three lemmas.  

\begin{lemma}\label{lemcon}Let $Z\sim \mathrm{SVG}(r,\sigma,0)$.  Then, for any random variable $W$,
\begin{equation}\label{bbnn}\mathbb{P}(a\leq W\leq b)\leq C_{r,\sigma,b-a}+2d_{\mathrm{K}}(W,Z),
\end{equation} 
where
\begin{equation*}C_{r,\sigma,\alpha}=\begin{cases}\displaystyle \frac{\alpha}{2\sigma\sqrt{\pi}}\frac{\Gamma\big(\frac{r-1}{2}\big)}{\Gamma\big(\frac{r}{2}\big)}, &  r>1, \\
\displaystyle \frac{\alpha}{\pi\sigma}\bigg[1+\log\bigg(\frac{2\sigma}{\alpha}\bigg)\bigg], &  r=1, \\
\displaystyle \frac{\Gamma\big(\frac{1-r}{2}\big)}{\sqrt{\pi}2^{r}\Gamma\big(\frac{r}{2}+1\big)}\bigg(\frac{\alpha}{\sigma}\bigg)^r, &  0<r<1. \end{cases}
\end{equation*}
\end{lemma}

\begin{proof}
Clearly,
\[\mathbb{P}(a\leq W\leq b)\leq \mathbb{P}(a\leq Z\leq b)+2d_{\mathrm{K}}(W,Z).\]
Since, for all $r>0$ and $\sigma>0$, the $ \mathrm{SVG}(r,\sigma,0)$ density $p(x)$ is an increasing function of $x$ for $x<0$ and a decreasing function of $x$ for $x>0$, we have that
\begin{equation}\label{nnmm}\mathbb{P}(a\leq Z\leq b)\leq \int_{-(b-a)/2}^{(b-a)/2} p(x)\,\mathrm{d}x=2\int_0^{(b-a)/2} p(x)\,\mathrm{d}x.
\end{equation}
To obtain (\ref{bbnn}), we bound the integral on the right-hand side of (\ref{nnmm}), treating the cases $r>1$, $r=1$ and $0<r<1$ separately.  For $r>1$ we bound the density $p(x)$ by $\frac{1}{2\sigma\sqrt{\pi}}\Gamma(\frac{r-1}{2})/\Gamma(\frac{r}{2})$ using (\ref{pmutend}) and then compute the trivial integral; for $r=1$ we use inequality (\ref{uu2}); and for $0<r<1$ we use inequality (\ref{uu3}).  This yields (\ref{bbnn}), as required. 
\end{proof}

The next lemma follows immediately from the estimates of Theorem \ref{thm1} and Corollary \ref{cor3.22}, and the subsequent lemma is straightforward and we hence omit the proof. 

\begin{lemma}\label{hae}For any $a\in\mathbb{R}$ and any $\epsilon>0$, let
\begin{equation}\label{hae5}h_{a,\epsilon}(x):=\epsilon^{-1}\int_0^\epsilon \mathbf{1}(x+s\leq a)\,\mathrm{d}s.
\end{equation}
Let $f_{a,\epsilon}$ be the solution of the $\mathrm{SVG}(r,\sigma,0)$ Stein equation with test function $h_{a,\epsilon}$.  Define $h_{a,0}(x)=\mathbf{1}(x\leq a)$ and $f_{a,0}$ accordingly. Then
\begin{eqnarray}\label{hae1}\|f_{a,\epsilon}\|&\leq&\frac{1}{\sigma}\bigg(\frac{1}{r}+\frac{\pi\Gamma(\frac{r}{2})}{2\Gamma(\frac{r+1}{2})}\bigg), \\
\label{hae2}\|xf_{a,\epsilon}(x)\|&\leq& \frac{3}{2}+\frac{1}{2r}, \\
\label{hae3}\|xf_{a,\epsilon}'(x)\|&\leq& \frac{1}{\sigma}\bigg(1+\frac{1}{2r}\bigg), \\
\label{hae4}\sigma^2\|T_rf_{a,\epsilon}'\|&\leq& \frac{5}{2}+\frac{1}{2r}.
\end{eqnarray}
\end{lemma}

\begin{lemma}\label{lemsmo}Let $Z\sim \mathrm{SVG}(r,\sigma,0)$ and $W$ be a real-valued random variable.  Then, for any $\epsilon>0$,
\begin{equation*}d_{\mathrm{K}}(W,Z)\leq C_{r,\sigma,\epsilon}+\sup_{a\in\mathbb{R}}|\mathbb{E}h_{a,\epsilon}(W)-\mathbb{E}h_{a,\epsilon}(Z)|,
\end{equation*}
where $C_{r,\sigma,\epsilon}$ is defined as in Lemma \ref{lemcon} and $h_{a,\epsilon}$ is defined as in Lemma \ref{hae}.
\end{lemma}

\begin{theorem} \label{jazzz} Let $W$ be a mean zero random variable with variance $0<r\sigma^2<\infty$.  Suppose that $(W,W^{V_r})$ is given on a joint probability space so that $W^{V_r}$ has the $W$-centered equilibrium distribution of order $r$.  Then  
\begin{equation}\label{dfgh1}d_{\mathrm{K}}(W,Z)\leq \bigg(2+\frac{3}{r}+\frac{\pi\Gamma\big(\frac{r}{2}\big)}{\Gamma\big(\frac{r+1}{2}\big)}\bigg)\frac{\beta}{\sigma}+\frac{5}{2}C_{r,\sigma,4\beta}+\bigg(10+\frac{2}{r}\bigg)\mathbb{P}(|W-W^{V_r}|>\beta),
\end{equation}
where $C_{r,\sigma,4\beta}$ is defined as in Lemma \ref{lemcon}.  Also,
\begin{equation}\label{dk76}d_{\mathrm{K}}(W^{V_r},Z)\leq  \bigg(1+\frac{3}{2r}+\frac{\pi\Gamma\big(\frac{r}{2}\big)}{2\Gamma\big(\frac{r+1}{2}\big)}\bigg)\frac{\beta}{\sigma}+\bigg(3+\frac{1}{r}\bigg)\mathbb{P}(|W-W^{V_r}|>\beta).
\end{equation}
Suppose in addition that $\mathbb{E}|W|^3<\infty$.  Then 
\begin{eqnarray}\label{zezozr}d_{\mathrm{W}}(W,Z) &\leq&\frac{9}{4}\bigg(5+\frac{1}{r+1}\bigg)\mathbb{E}|W-W^{V_r}|, \\
\label{zezozr1}d_{\mathrm{W}}(W^{V_r},Z) &\leq&\frac{1}{4}\bigg(41+\frac{9}{r+1}\bigg)\mathbb{E}|W-W^{V_r}|, \\
\label{zezozr2} d_{\mathrm{K}}(W^{V_r},Z) &\leq&\frac{1}{\sigma}\bigg(1+\frac{3}{2r}+\frac{\pi\Gamma\big(\frac{r}{2}\big)}{2\Gamma\big(\frac{r+1}{2}\big)}\bigg)\mathbb{E}|W-W^{V_r}|.
\end{eqnarray}
\end{theorem}

\begin{proof}For this proof, we shall write $\kappa=d_{\mathrm{K}}(W,Z)$. Let $\Delta:= W -W^{V_r}$. Define $I_1 := \mathbf{1}(|\Delta| \leq \beta)$; note that $W^{V_r}$ may not have finite second moment.  Let $f$ be the solution of the $\mathrm{SVG}(r,\sigma,0)$ Stein equation with test function $h_{a,\epsilon}$, as defined in (\ref{hae5}).  Then $\mathbb{E}T_rf'(W^{V_r})$ is well defined, because $\|T_rf'\|<\infty$ (see Lemma \ref{hae}), and we have 
\begin{align*}\mathbb{E}[\sigma^2T_rf'(W)-Wf(W)]&=\sigma^2\mathbb{E}[I_1(T_rf'(W)-T_rf'(W^{V_r}))]\\
&\quad+\sigma^2\mathbb{E}[(1-I_1)(T_rf'(W)-T_rf'(W^{V_r}))]\\
&=:J_1+J_2.
\end{align*}
Using (\ref{hae4}) gives $|J_2|\leq2\times\big(\frac{5}{2}+\frac{1}{2r}\big)\mathbb{P}(|\Delta|>\beta)$.  Arguing as we did at the start of the proof of Corollary \ref{cor3.22} to obtain the second equality, and then using inequalities (\ref{hae1}) and (\ref{hae3}) and Lemma \ref{lemcon} in the last step gives
\begin{align*}J_1&=\sigma^2\mathbb{E}\bigg[I_1\int_0^\Delta (T_rf')'(W+t)\,\mathrm{d}t\bigg] \\
&=\mathbb{E}\bigg[I_1\int_0^\Delta \big\{(W+t)f'(W+t)+f(W+t)-\epsilon^{-1}\mathbf{1}(a-\epsilon\leq W+t\leq a)\big\}\,\mathrm{d}t\bigg] \\
&\leq \big(\|xf'(x)\|+\|f\|\big)\mathbb{E}|I_1\Delta|+\epsilon^{-1}\int_{-\beta}^0\mathbb{P}(a-\epsilon\leq W+t\leq a)\,\mathrm{d}t \\
&\leq \bigg(1+\frac{3}{2r}+\frac{\pi\Gamma\big(\frac{r}{2}\big)}{2\Gamma\big(\frac{r+1}{2}\big)}\bigg)\frac{\beta}{\sigma}+\beta\epsilon^{-1} C_{r,\sigma,\epsilon}+2\beta\epsilon^{-1}\kappa.
\end{align*}
Similarly,
\begin{align*}J_1&\geq -\big(\|xf'(x)\|+\|f\|\big)\mathbb{E}|I_1\Delta|-\epsilon^{-1}\int_{-\beta}^0\mathbb{P}(a-\epsilon\leq W+t\leq a)\,\mathrm{d}t \\
&\geq  -\bigg(1+\frac{3}{2r}+\frac{\pi\Gamma\big(\frac{r}{2}\big)}{2\Gamma\big(\frac{r+1}{2}\big)}\bigg)\frac{\beta}{\sigma}-\beta\epsilon^{-1} C_{r,\sigma,\epsilon}-2\beta\epsilon^{-1}\kappa,
\end{align*}
and so we conclude that
\begin{equation*}|J_1|\leq \bigg(1+\frac{3}{2r}+\frac{\pi\Gamma\big(\frac{r}{2}\big)}{2\Gamma\big(\frac{r+1}{2}\big)}\bigg)\frac{\beta}{\sigma}+\beta\epsilon^{-1} C_{r,\sigma,\epsilon}+2\beta\epsilon^{-1}\kappa.
\end{equation*}
Using Lemma \ref{lemsmo} and taking $\epsilon=4\beta$ now gives
\begin{align*}\kappa&\leq \bigg(5+\frac{1}{r}\bigg)\mathbb{P}(|\Delta|>\beta)+\bigg(1+\frac{3}{2r}+\frac{\pi\Gamma\big(\frac{r}{2}\big)}{2\Gamma\big(\frac{r+1}{2}\big)}\bigg)\frac{\beta}{\sigma}+(1+\beta\epsilon^{-1}) C_{r,\sigma,\epsilon}\\
&\quad+2\beta\epsilon^{-1}\kappa \\
&\leq \bigg(5+\frac{1}{r}\bigg)\mathbb{P}(|\Delta|>\beta)+\bigg(1+\frac{3}{2r}+\frac{\pi\Gamma\big(\frac{r}{2}\big)}{2\Gamma\big(\frac{r+1}{2}\big)}\bigg)\frac{\beta}{\sigma}+\frac{5}{4} C_{r,\sigma,4\beta}+\frac{1}{2}\kappa,
\end{align*}
whence on solving for $\kappa$ yields (\ref{dfgh1}).

Now let us prove (\ref{dk76}).  We can write
\begin{align*}&\mathbb{E}[\sigma^2T_rf'(W^{V_r})-W^{V_r}f(W^{V_r})]=\mathbb{E}[Wf(W)-W^{V_r}f(W^{V_r})] \\
&\quad=\mathbb{E}[I_1(Wf(W)-W^{V_r}f(W^{V_r}))]+\mathbb{E}[(1-I_1)(Wf(W)-W^{V_r}f(W^{V_r}))].
\end{align*} 
Taylor expanding, applying the triangle inequality to $\|xf'(x)+f(x)\|$, and using the estimates (\ref{hae1}), (\ref{hae2}) and (\ref{hae3}) then gives that
\begin{align*}&\mathbb{E}[\sigma^2T_rf'(W^{V_r})-W^{V_r}f(W^{V_r})] \\
&\quad\leq \|xf'(x)+f(x)\|\mathbb{E}|I_1\Delta|+2\|xf(x)\|\mathbb{P}(|\Delta|>\beta)  \\
&\quad\leq \frac{1}{\sigma}\bigg(1+\frac{3}{2r}+\frac{\pi\Gamma\big(\frac{r}{2}\big)}{2\Gamma\big(\frac{r+1}{2}\big)}\bigg)\beta+\bigg(3+\frac{1}{r}\bigg)\mathbb{P}(|\Delta|>\beta),
\end{align*}
which gives (\ref{dk76}).

Suppose now that $\mathbb{E}|W|^3<\infty$, which, by part (iv) of Proposition \ref{propsquare1}, ensures that $\mathbb{E}|W^{V_r}|<\infty$. Let $h\in\mathcal{H}_{\mathrm{W}}$.  Then
\begin{equation*}\mathbb{E}h(W)-\mathbb{E}h(Z)=\mathbb{E}[\sigma^2T_r f'(W)-Wf(W)]=\sigma^2\mathbb{E}[T_r f'(W)-T_r f'(W^{V_r})],
\end{equation*}
and by Taylor expansion, we have 
\begin{equation*}|\mathbb{E}h(W)-\mathbb{E}h(Z)|\leq\sigma^2\|(T_r f')'\|\mathbb{E}|W-W^{V_r}|.
\end{equation*}
On using the estimate (\ref{lemse}) we obtain (\ref{zezozr}), as required.  Also,
\begin{align}\big|\sigma^2\mathbb{E}\big[(T_rf')(W^{V_r})-W^{V_r}f(W^{V_r})\big]\big| &=\big|\mathbb{E}Wf(W)-\mathbb{E}W^{V_r}f(W^{V_r})\big|\nonumber \\
\label{alige}&\leq  \|xf'(x)+f(x)\|\mathbb{E}|W-W^{V_r}|.
\end{align}
Applying the estimates (\ref{thmfour}) and (\ref{thm1f0}) to (\ref{alige}) yields (\ref{zezozr1}), whilst applying the estimates (\ref{thmtwo}) and (\ref{vgf0}) yields (\ref{zezozr2}).
\end{proof}

\section{Applications}\label{sec5}

\subsection{Comparison of variance-gamma distributions}\label{sec5.1}

The following proposition quantifies the error in approximating a general VG distribution by a SVG distribution.  We refer the reader to \cite{ley} for a number of similar bounds for comparison of univariate distributions.  The proof provides an example under which the bounds on $\|(x-\mu)f^{(k)}(x)\|$, $k=0,1,2,3$, for the solution of the SVG Stein equation that were given in Theorem \ref{thm1} prove useful.  This application also serves as a simple example in which the inequalities of Proposition \ref{prop1} are suboptimal.  

\begin{proposition}\label{proppl}Let $X\sim \mathrm{VG}(r_1,\theta_1,\sigma_1,\mu_1)$ and $Y\sim\mathrm{SVG}(r_2,\sigma_2,\mu_2)$.  Then
\begin{align}d_{\mathrm{W}}(X,Y)&\leq \frac{9}{2}\bigg(1+\frac{1}{2(r_2+1)}\bigg)\frac{|\sigma_1^2-\sigma_2^2|}{\sigma_2}\nonumber \\
&\quad+\frac{9}{2\sigma_2}\bigg(\frac{1}{r_2+1}+\frac{\pi\Gamma\big(\frac{r_2+1}{2}\big)}{2\Gamma\big(\frac{r_2}{2}+1\big)}\bigg)\big(|\sigma_1^2r_1-\sigma_2^2r_2|+2|\theta_1(\mu_1-\mu_2)|\big)\nonumber \\
\label{dov}&\quad+\bigg(\frac{7}{2}+\frac{9\sigma_1^2}{\sigma_2^2(r_2+1)}\bigg)|\mu_1-\mu_2|+\bigg(\frac{7r_1}{2}+\frac{27}{2}+\frac{9}{2(r_2+1)}\bigg)|\theta_1|.
\end{align}
Suppose now that $\mu_1=\mu_2$.  Then
\begin{align}\label{dov1}d_{\mathrm{K}}(X,Y)\leq \frac{1}{2}\bigg(9+\frac{1}{r_2}\bigg)\bigg|1-\frac{\sigma_1^2}{\sigma_2^2}\bigg|+2\bigg|1-\frac{\sigma_1^2r_1}{\sigma_2^2r_2}\bigg| +\frac{|\theta_1|}{\sigma_2}\bigg(2+\frac{r_1+1}{r_2}+\frac{\pi r_1\Gamma\big(\frac{r_2}{2}\big)}{2\Gamma\big(\frac{r_2+1}{2}\big)}\bigg).
\end{align}
\end{proposition}

\begin{remark}\label{remopl}The function $h(x)=x$ is in the class $\mathcal{H}_{\mathrm{W}}$.  Therefore
\[d_{\mathrm{W}}(X,Y)\geq |\mathbb{E}X-\mathbb{E}Y|=|r_1\theta_1+\mu_1-\mu_2|.\]
When $\mu_1=\mu_2$, this lower bound is equal to $r_1|\theta_1|$, and so there exist constants $c>0$ and $C>0$ independent of $\theta_1$ such that $c|\theta_1|\leq d_{\mathrm{W}}(X,Y)\leq C|\theta_1|$, if in addition $r_1=r_2$ and $\sigma_1=\sigma_2$.  Comparing with the Kolmogorov bound (\ref{dov1}), we see that the inequalities of Proposition \ref{prop1} are suboptimal in this application. 
\end{remark}

\begin{proof}Let $\mathcal{A}_{r,\theta,\sigma,\mu}$ denote the differential operator on the left-hand side of the VG Stein equation (\ref{377}). Suppose that $h:\mathbb{R}\rightarrow\mathbb{R}$ is either bounded or Lipschitz.  Let $f$ be the solution of the $\mathrm{SVG}(r_2,\sigma_2,\mu_2)$ Stein equation.  Then 
\begin{align}\mathbb{E}h(X)-\mathbb{E}h(Y)&=\mathbb{E}[\mathcal{A}_{r_2,0,\sigma_2,\mu_2}f(X)]\nonumber \\
\label{wizlr}&=\mathbb{E}[\mathcal{A}_{r_2,0,\sigma_2,\mu_2}f(X)-\mathcal{A}_{r_1,\theta_1,\sigma_1,\mu_1}f(X)].
\end{align}
That $\mathbb{E}[\mathcal{A}_{r_1,\theta_1,\sigma_1,\mu_1}f(X)]=0$ follows from the assumptions on $h$, the estimates of Theorem \ref{thm1}, and Lemma 3.1 of \cite{gaunt vg}.  Firstly, we prove (\ref{dov}).  Suppose $h\in\mathcal{H}_{\mathrm{W}}$.  Then, from (\ref{wizlr}), 
\begin{align}&|\mathbb{E}h(X)-\mathbb{E}h(Y)|\nonumber\\
&=\big|\mathbb{E}\big[\sigma_1^2(X-\mu_1)f''(X)+(\sigma_1^2r_1+2\theta_1(X-\mu_1)f'(X)+(r_1\theta_1-(X-\mu_1))f(X)\nonumber \\
&\quad-\sigma_2^2(X-\mu_2)f''(X) -\sigma_2^2r_2 f'(X)+(X-\mu_2)f(X)]\big|\nonumber \\
&=\big|\mathbb{E}\big[(\sigma_1^2-\sigma_2^2)(X-\mu_2)f''(X)+\sigma_1^2(\mu_2-\mu_1)f''(X) +(\sigma_1^2r_1-\sigma_2^2r_2)f'(X)\nonumber \\
&\quad+2\theta_1(X-\mu_2)f'(X)+2\theta_1(\mu_2-\mu_1)f'(X)+r_1 \theta_1 f(X)+(\mu_1-\mu_2)f(X)\big]\big|\nonumber\\
&\leq |\sigma_1^2-\sigma_2^2|\|(x-\mu_2)f''(x)\|+\sigma_1^2|\mu_1-\mu_2|\|f''\| +\big(|\sigma_1^2r_1-\sigma_2^2 r_2|+2|\theta_1(\mu_1-\mu_2)|\big)\|f'\|\nonumber\\
\label{near3}&\quad+2|\theta_1|\|(x-\mu_2)f'(x)\|+(r_1|\theta_1|+|\mu_1-\mu_2|)\|f\|.
\end{align}
Using the estimates of Theorem \ref{thm1} (with $\|h'\|\leq1$) to bound (\ref{near3}) yields (\ref{dov}).

Now suppose that $\mu_1=\mu_2$. Take $h_z(x)=\mathbf{1}(x\leq z)$.  On using the estimates of Corollary \ref{cor345} to bound (\ref{near3}), we obtain (\ref{dov1}), as required.
\end{proof}

\subsection{Malliavin-Stein method for symmetric variance-gamma approximation}\label{sec5.2}

In recent years, one of the most significant applications of Stein's method has been to Gaussian analysis on Wiener space.  This body of research was initiated by \cite{np09}, in which Stein's method and Malliavin calculus are combined to derive a quantitative ``fourth moment" theorem for the normal approximation of a sequence of random variables living in a fixed Wiener chaos.

In a recent work \cite{eichelsbacher}, the Malliavin-Stein method was extended to the VG distribution. Here, we obtain explicit constants in some of the main results (in the SVG case) of \cite{eichelsbacher}, these being six moment theorems for the SVG approximation of double Wiener-It\^{o} integrals. Our results also fix a technical issue in that the Wasserstein distance bounds stated in \cite{eichelsbacher} had only been proven in the weaker bounded Wasserstein distance (at the time of \cite{eichelsbacher} the bounds for the solution of the Stein equation in the literature \cite{gaunt thesis, gaunt vg} had a dependence on the test function $h$ such that this was the best that could be achieved).

Let us first introduce some notation; see the book \cite{np12} for a more detailed discussion.  Let $\mathbb{D}^{p,q}$ be the Banach space of all functions in $L^q(\gamma)$, where $\gamma$ is the standard Gaussian measure, whose Malliavin derivatives up to order $p$ also belong to $L^q(\gamma)$.  Let $\mathbb{D}^\infty$ be the class of infinitely many times Malliavin differentiable random variables.  We introduce the so-called $\Gamma$-operators $\Gamma_j$ \cite{np10}.  For a random variable $F\in\mathbb{D}^\infty$, we define $\Gamma_1(F)=F$ and, for every $j\geq2$,
\[\Gamma_j(F)=\langle DF, -DL^{-1}\Gamma_{j-1}(F)\rangle_{\mathfrak{H}}.\]
Here $D$ is the Malliavin derivative, $L^{-1}$ is the pseudo-inverse of the infinitesimal generator of the Ornstein-Uhlenbeck semi-group, and $\mathfrak{H}$ is a real separable Hilbert space.  Finally, for $f\in \mathfrak{H}^{\odot 2}$, we write $I_2(f)$ for the double Wiener-It\^{o} integral of $f$.



\begin{theorem}Let $F\in\mathbb{D}^{2,4}$ be such that $\mathbb{E}F=0$ and let $Z\sim \mathrm{SVG}(r,\sigma,0)$.  Then
\begin{align}\label{etqu}d_{\mathrm{W}}(F,Z)&\leq\frac{9}{\sigma^2(r+1)}\mathbb{E}|\sigma^2 F-\Gamma_3(F)| +\frac{9}{2\sigma}\bigg(\frac{1}{r+1}+\frac{\pi\Gamma\big(\frac{r+1}{2}\big)}{2\Gamma\big(\frac{r}{2}+1\big)}\bigg)|r\sigma^2-\mathbb{E}[\Gamma_2(F)]|.
\end{align}
If in addition $F\in\mathbb{D}^{3,8}$, then $\Gamma_3(F)$ is square-integrable and
\begin{equation}\label{etpr}\mathbb{E}|\sigma^2 F-\Gamma_3(F)|\leq\big(\mathbb{E}[(\sigma^2 F-\Gamma_3(F))^2]\big)^{\frac{1}{2}}.
\end{equation}
\end{theorem}

\begin{proof}Let $f:\mathbb{R}\rightarrow\mathbb{R}$ be twice differentiable with bounded first and second derivative.  Then it was shown in the proof of Theorem 4.1 of \cite{eichelsbacher} that
\begin{align}&\big|\mathbb{E}\big[\sigma^2Ff''(F)+\sigma^2 rf'(F)-Ff(F)\big]\big|\nonumber\\
\label{polk}&\quad=\big|\mathbb{E}\big[f''(F)(\sigma^2 F-\Gamma_3(F))+f'(F)(r\sigma^2-\mathbb{E}[\Gamma_2(F)])\big]\big|  \\
&\quad\leq \|f''\|\mathbb{E}|\sigma^2 F-\Gamma_3(F)|+\|f'\|\mathbb{E}|r\sigma^2-\mathbb{E}[\Gamma_2(F)]|\nonumber.
\end{align} 
If $h\in\mathcal{H}_{\mathrm{W}}$, then the solution $f$ of the $\mathrm{SVG}(r,\sigma,0)$ Stein equation is twice differentiable with bounded first and second derivatives.  Using the estimates (\ref{thm1f2}) and (\ref{thm1f1}) of Theorem \ref{thm1} to bound $\|f''\|$ and $\|f'\|$ then yields (\ref{etqu}).  Inequality (\ref{etpr}) is justified in \cite{eichelsbacher}.
\end{proof}

\begin{corollary}\label{cor5.4}Let $F_n=I_2(f_n)$ with $f_n\in \mathfrak{H}^{\odot 2}$, $n\geq1$.  Also, let $Z\sim\mathrm{SVG}(r,\sigma,0)$ and assume that $\mathbb{E}[F_n^2]=r\sigma^2$.  Then
\begin{equation}\label{doww}d_{\mathrm{W}}(F_n,Z)\leq \frac{9}{\sigma^2(r+1)}\bigg(\!\frac{1}{120}\kappa_6(F_n)-\frac{\sigma^2}{3}\kappa_4(F_n)+\frac{1}{4}(\kappa_3(F_n))^2+\sigma^4\kappa_2(F_n)\!\bigg)^{\frac{1}{2}}.
\end{equation}
\end{corollary}

\begin{proof}It is a standard result that $\mathbb{E}[\Gamma_2(F_n)]=\kappa_2(F_n)$ (see Lemma 4.2 and Theorem 4.3 of \cite{np10}), and it was shown in the proof of Theorem 5.8 of \cite{eichelsbacher} that
\begin{equation*}\mathbb{E}[(\sigma^2 F_n-\Gamma_3(F_n))^2]=\frac{1}{120}\kappa_6(F_n)-\frac{\sigma^2}{3}\kappa_4(F_n)+\frac{1}{4}(\kappa_3(F_n))^2+\sigma^4\kappa_2(F_n).
\end{equation*}
Inserting these formulas into (\ref{etqu}) yields (\ref{doww}), as required.
\end{proof}

\begin{remark}One can obtain Kolmogorov distance bounds by applying Proposition \ref{prop1} to the bound (\ref{doww}). However, these bounds are unlikely to be of optimal order.  Unlike for normal approximation, for which an optimal rate of convergence in Kolmogorov distance has been obtained \cite{np15}, there is a technical difficulty for SVG approximation because the first derivative of the solution $f_z$ of the $\mathrm{SVG}(r,\sigma,0)$ Stein equation with test function $h_z(x)=\mathbf{1}(x\leq z)$ has a discontinuity at the origin when $z=0$ (see Proposition \ref{disclem}).  We can, however, bound the expression using the inequalities (\ref{thmthree}) for $\|xf''(x)\|$ and (\ref{vgf1}) for $\|f'\|$ to obtain the bound
\begin{align*}d_{\mathrm{K}}(F,Z)&\leq \frac{1}{2\sigma^2}\bigg(9+\frac{1}{r}\bigg)\mathbb{E}\bigg|\sigma^2 -\frac{\Gamma_3(F)}{F}\bigg| +\frac{2}{\sigma^2 r}|r\sigma^2-\mathbb{E}[\Gamma_2(F)]| \\
&\leq \frac{1}{2\sigma^2}\bigg(9+\frac{1}{r}\bigg)\bigg\{\mathbb{E}\bigg[\bigg(\sigma^2 -\frac{\Gamma_3(F)}{F}\bigg)^2\bigg]\bigg\}^{\frac{1}{2}} +\frac{2}{\sigma^2r}|r\sigma^2-\mathbb{E}[\Gamma_2(F)]|,
\end{align*} 
provided the expectations exist.  However, there are no formulas in the literature for the expectations $\mathbb{E}[\Gamma_3(F)/F]$ and $\mathbb{E}[(\Gamma_3(F))^2/F^2]$ (when they exist), and it is unlikely they could be expressed solely in terms of lower order cumulants of $F$.
\end{remark}

\subsection{Binary sequence comparison}\label{sec5.3}


Here we consider an application of Theorem \ref{jazzz} to binary sequence comparison.  This a special case of a more general problem of word sequence comparison, which is of importance to biological sequence comparison.  One way of comparing sequences uses $k$-tuples (a sequence of letters of length $k$).  If two sequences are closely related, we would expect their $k$-tuple content to be similar. A statistic for sequence comparison based on $k$-tuple content, known as the $D_2$ statistic, was suggested by \cite{bla} (see \cite{waterman} for further statistics based on $k$-tuple content).  Letting $\mathcal{A}$ denote an alphabet of size $d$, and $X_{\mathbf{w}}$ and $Y_{\mathbf{w}}$ the number of occurrences of the word $\mathbf{w}\in\mathcal{A}^k$ in the first and second sequences, respectively, then the $D_2$ statistic is defined by
\[D_2=\sum_{\mathbf{w}\in\mathcal{A}^k}X_{\mathbf{w}}Y_{\mathbf{w}}.\] 

Due to the complicated dependence structure (for a detailed account see \cite{lothaire}) approximating the asymptotic distribution of $D_2$ is a difficult problem.  However, for certain parameter regimes $D_2$ has been shown to be asymptotically normal and Poisson \cite{lippert}.  

We now consider the case of an alphabet of size $2$ with comparison based on the content of $1$-tuples.  We suppose that the sequences are of length $m$ and $n$, the alphabet is $\{0,1\}$, and $\mathbb{P}(0 \mbox{ appears})=\mathbb{P}(1 \mbox{ appears})=\frac{1}{2}$.  Denoting the number of occurrences of $0$ in the two sequences by $X$ and $Y$, then
\[D_2=XY+(m-X)(n-Y).\]
Clearly, $X$ and $Y$ are independent binomial variables with expectations $\frac{m}{2}$ and $\frac{n}{2}$.  Straightforward calculations (see \cite{lippert}) show that $\mathbb{E}D_2=\frac{mn}{2}$ and $\mathrm{Var}(D_2)=\frac{mn}{4}$ and the standardised $D_2$ statistic can be written as
\begin{equation} \label{cox apple} W=\frac{D_2-\mathbb{E}D_2}{\sqrt{\mathrm{Var}(D_2)}} =\bigg(\frac{X-\frac{m}{2}}{\sqrt{\frac{m}{4}}}\bigg)\bigg(\frac{Y-\frac{n}{2}}{\sqrt{\frac{n}{4}}}\bigg). 
\end{equation}
By the central limit theorem, $(X-\frac{m}{2})/\sqrt{\frac{m}{4}}$ and $(Y-\frac{n}{2})/\sqrt{\frac{n}{4}}$ are approximately $N(0,1)$ distributed, and so $W$ has an approximate $\mathrm{SVG}(1,1,0)$ distribution.  In \cite{gaunt vg}, a $O(m^{-1}+n^{-1})$ bound for the rate of convergence was given in a smooth test function metric (which requires the test function to be three times differentiable).  In Theorem \ref{d2thm} below we use Theorem \ref{jazzz} to obtain bounds in the more usual Wasserstein and Kolmogorov metrics.  Our rate of convergence is slower, but we do quantify the approximation in stronger metrics.  We will first need to prove the following theorem. 

\begin{theorem}\label{jazzz4}Suppose $X_1,\ldots,X_m$ are i.i.d$.$ and $Y_1,\ldots,Y_n$ are i.i.d$.$, with $\mathbb{E}X_1=\mathbb{E}Y_1=0$, $\mathbb{E}X_1^2=\mathbb{E}Y_1^2$ and $\mathbb{E}|X_1|^3<\infty$ and $\mathbb{E}|Y_1|^3<\infty$.  Let $W_1=\frac{1}{\sqrt{m}}\sum_{i=1}^m X_i$ and $W_2=\frac{1}{\sqrt{n}}\sum_{i=1}^n Y_i$ and set $W=W_1 W_2$.  Let $Z\sim\mathrm{SVG}(1,1,0)$.  Then
\begin{equation}\label{jazzwas}d_{\mathrm{W}}(W,Z)\leq 20.11\bigg(\frac{1}{\sqrt{m}}+\frac{1}{\sqrt{n}}\bigg)\mathbb{E}|X_1|^3\mathbb{E}|Y_1|^3.
\end{equation}
If in addition $\mathbb{E}X_1^3=\mathbb{E}Y_1^3=0$ and $\mathbb{E}X_1^4<\infty$ and $\mathbb{E}Y_1^4<\infty$, then
\begin{align}\label{jazzkol}d_{\mathrm{K}}(W,Z)&\leq \bigg\{44.33+2.02\bigg[\log\bigg(\frac{1}{\mathbb{E}X_1^4\mathbb{E}Y_1^4}\bigg)+\log\bigg(\frac{mn}{m+n}\bigg)\bigg]\bigg\}\bigg(\frac{1}{m}+\frac{1}{n}\bigg)^{\frac{1}{3}}\big(\mathbb{E}X_1^4\mathbb{E}Y_1^4\big)^{\frac{1}{3}}.
\end{align}
\end{theorem}

\begin{remark}The rate of convergence in Kolmogorov distance bound (\ref{jazzkol}) is unlikely to be of optimal order, but is better than the $O\big(m^{-\frac{1}{4}}\log(m)+n^{-\frac{1}{4}}\log(n)\big)$ rate that would result from simply applying Proposition \ref{prop1} to (\ref{jazzwas}).  A reasonable conjecture is that the optimal rate is $O(m^{-\frac{1}{2}}+n^{-\frac{1}{2}})$.
\end{remark}

\begin{proof}Since $Z\sim \mathrm{SVG}(1,1,0)$, we will apply Theorem \ref{jazzz} with $r=1$, for which $W^{V_1}=W^{*(2)}$, the $W$-zero bias transformation of order 2.  We begin by collecting some useful properties of this distributional transformation. In \cite{gaunt pn}, the following construction is given: $W^{*(2)}=\frac{1}{\sqrt{mn}}W_1^*W_2^*$.  Since $W_1$ and $W_2$ are sums of independent random variables, we have by part (v) of Lemma 2.1 of \cite{goldstein} that $W_1^*=W_1-\frac{X_I}{\sqrt{m}}+\frac{X_I^*}{\sqrt{m}}$ and $W_2^*=W_2-\frac{Y_J}{\sqrt{n}}+\frac{Y_J^*}{\sqrt{n}}$, where $I$ and $J$ are chosen uniformly from $\{1,\ldots,m\}$ and $\{1,\ldots,n\}$ respectively.  It was shown in the proofs of Corollaries 4.1 and 4.2 of \cite{gaunt pn} that
\begin{equation}\label{imjk}\mathbb{E}|W-W^{*(2)}|\leq \frac{13}{8}\bigg(\frac{1}{\sqrt{m}}+\frac{1}{\sqrt{n}}\bigg)\mathbb{E}|X_1|^3\mathbb{E}|Y_1|^3
\end{equation}
and, if $\mathbb{E}X_1^3=\mathbb{E}Y_1^3=0$,
\begin{equation}\label{plug6}\mathbb{E}[(W-W^{*(2)})^2]\leq \frac{20}{3}\bigg(\frac{1}{m}+\frac{1}{n}\bigg)\mathbb{E}X_1^4\mathbb{E}Y_1^4.
\end{equation}
The assumption $\mathbb{E}X_1^3=\mathbb{E}Y_1^3=0$ implies $\mathbb{E}X_1^*=\mathbb{E}Y_1^*=0$ (\cite{goldstein}, part (iv) of Lemma 2.1), which allowed \cite{gaunt pn} to obtain the $O(m^{-1}+n^{-1})$ rate in (\ref{plug6}).

The bound (\ref{jazzwas}) is immediate from (\ref{imjk}) and (\ref{zezozr}):
\begin{equation*}d_{\mathrm{W}}(W,Z)\leq \frac{99}{8}\mathbb{E}|W-W^{*(2)}|\leq \frac{1287}{64}\bigg(\frac{1}{\sqrt{m}}+\frac{1}{\sqrt{n}}\bigg)\mathbb{E}|X_1|^3\mathbb{E}|Y_1|^3,
\end{equation*}
and $\frac{1287}{64}=20.11$.

Now we prove (\ref{jazzkol}).  We begin by setting $r=\sigma=1$ in (\ref{dfgh1}), using that $\Gamma(\frac{1}{2})=\sqrt{\pi}$, and applying Markov's inequality to obtain
\begin{align*}d_{\mathrm{K}}(W,Z)&\leq\bigg\{5+\pi^{3/2}+\frac{10}{\pi}\bigg[1+\log\bigg(\frac{1}{2}\bigg)\bigg]\bigg\}\beta+\frac{10}{\pi}\log\bigg(\frac{1}{\beta}\bigg)+12\mathbb{P}(|W-W^{*(2)}|>\beta) \\
&\leq\bigg\{11.55+3.19\log\bigg(\frac{1}{\beta}\bigg)\bigg\}\beta +12\frac{\mathbb{E}[(W-W^{*(2)})^2]}{\beta^2}.
\end{align*}
Setting $\beta=\big(\mathbb{E}[(W-W^{*(2)})^2]\big)^\frac{1}{3}$ gives
\begin{equation}\label{plug7}d_{\mathrm{K}}(W,Z)\leq \bigg\{23.55+1.07\log\bigg(\frac{1}{\mathbb{E}[(W-W^{*(2)})^2]}\bigg)\bigg\}\big(\mathbb{E}[(W-W^{*(2)})^2]\big)^\frac{1}{3}.
\end{equation}
Substituting (\ref{plug6}) into (\ref{plug7}) and simplifying then yields (\ref{jazzkol}).
\end{proof}


\begin{theorem}\label{d2thm}Let $W$ be the standardised $D_2$ statistic, as defined in (\ref{cox apple}), based on 1-tuple content, for uniform i.i.d$.$ binary sequences of lengths $m$ and $n$. Let $Z\sim\mathrm{SVG}(1,1,0)$.  Then
\begin{eqnarray*}d_{\mathrm{W}}(W,Z)&\leq& 20.11\bigg(\frac{1}{\sqrt{m}}+\frac{1}{\sqrt{n}}\bigg), \\
d_{\mathrm{K}}(W,Z)&\leq& \bigg\{44.33+2.02\log\bigg(\frac{mn}{m+n}\bigg)\bigg\}\bigg(\frac{1}{m}+\frac{1}{n}\bigg)^{\frac{1}{3}}.
\end{eqnarray*}
\end{theorem}

\begin{proof}Let $\mathbb{I}_i$ and $\mathbb{J}_i$ be the indicator random variables that letter $0$ occurs at position $i$ in the first and second sequences, respectively.  Then $X=\sum_{i=1}^m\mathbb{I}_i$ and $Y=\sum_{j=1}^n\mathbb{J}_j$.  We may then write
\[W=\bigg(\frac{X-\frac{m}{2}}{\sqrt{\frac{m}{4}}}\bigg)\bigg(\frac{Y-\frac{n}{2}}{\sqrt{\frac{n}{4}}}\bigg)=\bigg(\frac{1}{\sqrt{m}}\sum_{i=1}^{m}X_i\bigg)\bigg(\frac{1}{\sqrt{n}}\sum_{j=1}^{n}Y_j\bigg),\]
where $X_i=2(\mathbb{I}_i-\frac{1}{2})$ and $Y_j=2(\mathbb{J}_j-\frac{1}{2})$.  The $X_i$ and $Y_j$ are all independent with zero mean and unit variance.  Also, $\mathbb{E}X_1^3=\mathbb{E}Y_1^3=0$, $\mathbb{E}|X_1|^3=\mathbb{E}|Y_1|^3=1$ and $\mathbb{E}X_1^4=\mathbb{E}Y_1^4=1$, and the result now follows from Theorem \ref{jazzz4}.  
\end{proof}

\subsection{Random sums}\label{sec5.4}

Let $X_1,X_2,\ldots$ be i.i.d$.$, positive, non-degenerate random variables with unit mean.  Let $N_p$ be a $\mathrm{Geo}(p)$ random variable with $\mathbb{P}(N_p=k)=p(1-p)^{k-1}$, $k\geq1$, that is independent of the $X_i$.  Then, a well-known result of \cite{renyi} states that $p\sum_{i=1}^{N_p}X_i$ converges in distribution to an exponential distribution with parameter 1 as $p\rightarrow0$.  Geometric summation does indeed arise in a variety of settings; see \cite{k97}.  Stein's method was used by \cite{pekoz1} to obtain a quantitative generalisation of the result of \cite{renyi}.  If we alter the assumptions so that the $X_i$ have mean zero and finite non-zero variance, then $p^{\frac{1}{2}}\sum_{i=1}^{N_p}X_i$ converges to a Laplace distribution as $p\rightarrow0$; see \cite{toda} and \cite{pike}.  Recently, \cite{pike}, through the use of the centered equilibrium transformation, mirrored the approach of \cite{pekoz1} to obtain an explicit error bound in the bounded Wasserstein metric.  

In this section, we use Theorem \ref{jazzz} to obtain Wasserstein and Kolmogorov error bounds for the theorems of \cite{pike}.  Indeed, Theorems \ref{thmnbc} and \ref{thm555} below give Wasserstein and Kolmogorov distance bounds for the approximations of Theorems 1.3 and 4.4 of \cite{pike}, respectively.  The results of \cite{pekoz1} are also given in these metrics, and we follow their approach to obtain our Kolmogorov bounds.   For a random
variable $X$, we denote by distribution function by $F_X$ and its generalized inverse by $F_X^{-1}$.




\begin{theorem}\label{thmnbc}Let $N$ be a positive, integer valued random variables with $\mu=\mathbb{E}N<\infty$ and let $X_1,X_2,\ldots$ be a sequence of independent random variables, independent of $N$, with $\mathbb{E}X_i=0$ and $\mathbb{E}X_i^2=\sigma_i^2\in(0,\infty)$.  Set $\sigma^2=\frac{1}{\mu}\mathbb{E}\big[\big(\sum_{i=1}^NX_i\big)^2\big]=\frac{1}{\mu}\mathbb{E}\big[\sum_{i=1}^N\sigma_i^2\big]$.  Also, let $M$ be any positive, integer valued random variable, independent of the $X_i$, satisfying
\begin{equation*}\mathbb{P}(M=m)=\frac{\sigma_m^2}{\mu\sigma^2}\mathbb{P}(N\geq m), \quad m=1,2,\ldots.
\end{equation*}
Let $Z\sim \mathrm{Laplace}(0,\frac{\sigma}{\sqrt{2}})$.  Then, with $W=\frac{1}{\sqrt{\mu}}\sum_{i=1}^NX_i$, we have
\begin{equation}\label{wedf}d_{\mathrm{W}}(W, Z)\leq 12\mu^{-\frac{1}{2}}\big\{\mathbb{E}|X_M-X_M^L|+\sup_{i\geq1}\sigma_i\mathbb{E}\big[|N-M|^{\frac{1}{2}}\big]\big\}.
\end{equation}
Suppose further that $|X_i|\leq C$ for all $i$ and $|N-M|\leq K$.  Then
\begin{equation}\label{wdv}
d_{\mathrm{K}}(W, Z)\leq \frac{17.04}{\sigma\sqrt{\mu}}\Big\{\sup_{i\geq1}\|F_{X_i}^{-1}-F_{X_i^L}^{-1}\|+CK\Big\};
\end{equation}
if $K=0$, the same bound also holds for unbounded $X_i$.
\end{theorem}

\begin{proof}Since $Z\sim \mathrm{Laplace}(0,\frac{\sigma}{\sqrt{2}})=_d\mathrm{SVG}(2,\frac{\sigma}{\sqrt{2}},0)$, we will apply Theorem \ref{jazzz} with $r=2$, for which $W^{V_1}=W^{L}$, the $W$-centered equilibrium distribution.  For $W$ as defined in the statement of the theorem, it was shown in the proof of Theorem 4.4 of \cite{pike} that $W^L=\mu^{-\frac{1}{2}}\big(\sum_{i=1}^{M-1}+X_M^L\big)$.  Then
\begin{equation*}W^L-W=\mu^{-\frac{1}{2}}\bigg\{(X_M^L-X_M)+\mathrm{sgn}(M-N)\sum_{i=(M\wedge N)+1}^{N\vee M}X_i\bigg\}.
\end{equation*}
Plugging this into (\ref{zezozr}) (with $r=2$) and bounding $\mathbb{E}\big|\sum_{i=(M\wedge N)+1}^{N\vee M}X_i\big|\leq \sup_{i\geq1}\sigma_i\mathbb{E}\big[|N-M|^{\frac{1}{2}}\big]$ (see the proof of Theorem 4.4 of \cite{pike}) yields (\ref{wedf}).  Now, using (\ref{dfgh1}) and the formulas $\Gamma(\frac{1}{2})=\sqrt{\pi}$ and $\Gamma(\frac{3}{2})=\frac{\sqrt{\pi}}{2}$ gives that
\begin{align}d_{\mathrm{K}}(W, Z)&\leq \bigg(\frac{7}{2}+2\sqrt{\pi}\bigg)\frac{\beta\sqrt{2}}{\sigma}+\frac{5}{2}\cdot\frac{2\sqrt{2}\beta}{\sigma}+11\mathbb{P}(|W-W^L|>\beta)\nonumber \\
\label{wdvb}&=17.04\frac{\beta}{\sigma}+11\mathbb{P}(|W-W^L|>\beta).
\end{align}
Letting $\beta=\mu^{-\frac{1}{2}}\big\{\sup_{i\geq1}\|F_{X_i}^{-1}-F_{X_i^L}^{-1}\|+CK\big\}$, and using Strassen's theorem we obtain (\ref{wdv}) from (\ref{wdvb}), and the remark after (\ref{wdv}) follows similarly.
\end{proof}

\begin{theorem}\label{thm555}Let $X_1,X_2,\ldots$ be a sequence of independent random variables with $\mathbb{E}X_i=0$, $\mathbb{E}X_i^2=\sigma^2$, and let $N\sim\mathrm{Geo}(p)$ be independent of the $X_i$.  Let $W=p^{\frac{1}{2}}\sum_{i=1}^N X_i$ and let $Z\sim \mathrm{Laplace}(0,\frac{\sigma}{\sqrt{2}})$.  Then
\begin{equation}\label{wedfg}d_{\mathrm{K}}(W, Z)\leq 17.04\frac{p^{\frac{1}{2}}}{\sigma}\sup_{i\geq1}\|F_{X_i}^{-1}-F_{X_i^L}^{-1}\|.
\end{equation}
If in addition $\rho=\sup_{i\geq1}\mathbb{E}|X_i|^3<\infty$, then
\begin{equation}\label{rwrwa}d_{\mathrm{W}}(W,Z)\leq 12p^{\frac{1}{2}}\bigg(\sigma+\frac{\rho}{3\sigma^2}\bigg).
\end{equation}
The $O(p^{\frac{1}{2}})$ rate in (\ref{rwrwa}) is optimal.
\end{theorem}

\begin{remark}\label{rety}

(i) Theorem 1.3 of \cite{pike} gives the bound
\begin{equation}\label{rwrwatt}d_{\mathrm{BW}}(W,Z)\leq p^{\frac{1}{2}}(2\sqrt{2}+\sigma)\bigg(\sigma+\frac{\rho}{3\sigma^2}\bigg).
\end{equation}
which holds under the same conditions as (\ref{rwrwa}).  Aside from being given in a stronger metric, the bound (\ref{rwrwa}) has a theoretical advantage of having a multiplicative constant, 12, which is independent of $\sigma$, whereas (\ref{rwrwatt}) has a multiplicative constant $2\sqrt{2}+\sigma$.  The bound (\ref{rwrwatt}) has a smaller constant than (\ref{rwrwa}) when $\sigma<12-2\sqrt{2}$, whilst the constant is larger when $\sigma>12-2\sqrt{2}$. 

(ii) The argument used to prove the final assertion of Theorem \ref{thm555} also shows that the $O(p^{\frac{1}{2}})$ rate in (\ref{rwrwatt}) is optimal.

(iii) Suppose now that $\tau=\sup_{i\geq1}\mathbb{E}X_i^4<\infty$.  Then arguing as we did in the proof of Theorem \ref{jazzz4} would result in the alternative bound
\begin{equation}\label{pirc1}d_{\mathrm{K}}(W,Z)\leq Cp^{\frac{1}{3}}(1+\tau),
\end{equation}
where $C>0$ does not depend on $p$.  Thus, the dependence on $p$ is worse than in (\ref{wedfg}), but (\ref{pirc1}) may be preferable if $\sup_{i\geq1}\|F_{X_i}^{-1}-F_{X_i^L}^{-1}\|$ is difficult to compute or large.  The same remark applies to Theorem \ref{thmnbc}.

The quantity $\sup_{i\geq1}\|F_{X_i}^{-1}-F_{X_i^L}^{-1}\|$ can be easily bounded if the $X_i$ have finite support.  To see this, suppose that $X_1,X_2,\ldots$ are supported on a subset of the finite interval $[a,b]\subset\mathbb{R}$.  Theorem 3.2 of \cite{pike} (see also Proposition \ref{propsquare}) gives that $X^L=_d B_2 X^*$, where $B_2\sim\mathrm{Beta}(2,1)$ and $X^*$, the $X$-zero bias distribution, are independent.  But part (ii) of Lemma 2.1 of \cite{goldstein} tells us that the support of $X^*$ is the closed convex hull of the support of $X$, and since $B_r$ is supported on $[0,1]$ it follows that $X^L$ is supported on $[a,b]$.  We therefore have the bound $\sup_{i\geq1}\|F_{X_i}^{-1}-F_{X_i^L}^{-1}\|\leq b-a$.  

\end{remark}

\begin{proof}As noted by \cite{pike}, the assumptions on $N$ and the $X_i$ imply that $\mathcal{L}(M)=\mathcal{L}(N)$, so we can take $M=N$.  Inequality (\ref{wedfg}) is now immediate from (\ref{wdv}).  To obtain (\ref{rwrwa}), we note the inequality (see \cite{pike}) 
\begin{equation*}\mathbb{E}|X_N-X_N^L|\leq \mathbb{E}|X_1|+\sup_{i\geq1}\mathbb{E}|X_i^L|=\mathbb{E}|X_1|+\sup_{i\geq1}\frac{\mathbb{E}|X_i|^3}{3\sigma^2}\leq\sigma+\frac{\rho}{3\sigma^2},
\end{equation*}
where we used the Cauchy-Schwarz inequality.  Inequality (\ref{rwrwa}) now follows from (\ref{wedf}).

Finally, we prove that the $O(p^{\frac{1}{2}})$ rate in (\ref{rwrwa}) is optimal.  Suppose, in addition to the assumptions in the statement of the theorem, that $X_1,X_2,\ldots$ are i.i.d$.$ with moments of all order and $\mathbb{E}X_1^3\not=0$.  Consider the test function $h(x)=\sin(tx)$, $|t|\leq1$, which is in the class $\mathcal{H}_{\mathrm{W}}$.  We have $\mathbb{E}\sin(tZ)=0$.  We now consider the characteristic function $\varphi_W(t)=\mathbb{E}[\mathrm{e}^{\mathrm{i}tW}]$, and note the relation $\mathbb{E}\sin(tW)=\mathrm{Im}[\varphi_W(t)]$.  From the above, we have that $d_{\mathrm{W}}(W,Z)\geq |\mathrm{Im}[\varphi_W(t)]|$.  Recall that the probability generating function of $N\sim \mathrm{Geo}(p)$ is given by $G_N(s)=\frac{ps}{1-(1-p)s}$, $s<-\log(1-p)$.  Then 
\begin{equation}\label{refds}\varphi_W(t)=G_N\big(\varphi_{X_1}(p^{\frac{1}{2}}t)\big)=\frac{p\varphi_{X_1}(p^{\frac{1}{2}}t)}{1-(1-p)\varphi_{X_1}(p^{\frac{1}{2}}t)}.
\end{equation}
Now, since $\mathbb{E}X_1=0$ and $\mathbb{E}X_1^2=\sigma^2$, as $p\rightarrow0$,
\begin{equation}\label{refgh}\varphi_{X_1}(p^{\frac{1}{2}}t)=1-\frac{1}{2}pt^2\sigma^2-\frac{1}{6}\mathrm{i}p^{\frac{3}{2}}t^3\mathbb{E}X_1^3 +O(p^2).
\end{equation}
Plugging (\ref{refgh}) into (\ref{refds}) and performing a simple asymptotic analysis using the formula $\frac{1}{1+z}=1-z+O(|z|^2)$, $|z|\rightarrow0$, gives that $\mathrm{Im}[\varphi_W(t)]=-\frac{\frac{1}{6}p^{1/2}t^3\mathbb{E}X_1^3}{1+\sigma^2t^2/2}+O(p)$, 
and so the  $O(p^{\frac{1}{2}})$ rate cannot be improved.
\end{proof}

\section{Further proofs}\label{sec6}

\noindent\emph{Proof of Proposition \ref{disclem}.} As usual, we set $\sigma=1$ and $\mu=0$.  The solution of the $\mathrm{SVG}(r,1,0)$ Stein equation with test function $h_z(x)=\mathbf{1}(x\leq z)$ is then
\begin{align}f_z(x)&=-\frac{K_\nu(|x|)}{|x|^\nu}\int_0^x|t|^\nu I_\nu(|t|)[\mathbf{1}(t\leq z)-\mathbb{P}(Z\leq z)]\,\mathrm{d}t\nonumber \\
\label{ksoln}&\quad-\frac{I_\nu(|x|)}{|x|^\nu}\int_x^\infty |t|^\nu K_\nu(|t|)[\mathbf{1}(t\leq z)-\mathbb{P}(Z\leq z)]\,\mathrm{d}t.
\end{align}
Setting $z=0$ and differentiating (\ref{ksoln}) using (\ref{diff11}) and (\ref{diff22}) gives that
\begin{align*}f_0'(x)&=\frac{K_{\nu+1}(|x|)}{|x|^\nu}\mathrm{sgn}(x)\int_0^x|t|^\nu I_\nu(|t|)[\mathbf{1}(t\leq 0)-\tfrac{1}{2}]\,\mathrm{d}t \\
&\quad-\frac{I_{\nu+1}(|x|)}{|x|^\nu}\mathrm{sgn}(x)\int_x^\infty |t|^\nu K_\nu(|t|)[\mathbf{1}(t\leq 0)-\tfrac{1}{2}]\,\mathrm{d}t.
\end{align*}

We now note that, for all $\nu>-\frac{1}{2}$,
\begin{equation*}\lim_{x\rightarrow0}\bigg[\frac{I_{\nu+1}(|x|)}{|x|^\nu}\int_x^\infty |t|^\nu K_\nu(|t|)[\mathbf{1}(t\leq 0)-\tfrac{1}{2}]\,\mathrm{d}t\bigg]=0,
\end{equation*}
due to the asymptotic formula (\ref{Itend0}) and the  fact that $|t|^\nu K_\nu(|t|)$ is a constant multiple of the $\mathrm{SVG}(r,1,0)$ density meaning that the integral is bounded for all $x\in\mathbb{R}$.  Then
\begin{eqnarray*}f_0'(0+)&=&-\lim_{x\downarrow0}\bigg[\frac{K_{\nu+1}(x)}{2x^\nu}\int_0^x t^\nu I_\nu(t)\,\mathrm{d}t\bigg], \\
f_0'(0-)&=&-\lim_{x\uparrow0}\bigg[\frac{K_{\nu+1}(-x)}{2(-x)^\nu}\int_0^x(-t)^\nu I_\nu(-t)\,\mathrm{d}t\bigg] \\
&=&\lim_{x\uparrow0}\bigg[\frac{K_{\nu+1}(-x)}{2(-x)^\nu}\int_0^{-x}u^\nu I_\nu(u)\,\mathrm{d}u\bigg].
\end{eqnarray*}
On using the asymptotic formulas (\ref{Itend0}) and (\ref{Ktend0}), we obtain $f_0'(0+)=-\frac{1}{2(2\nu+1)}$ and $f_0'(0-)=\frac{1}{2(2\nu+1)}$, which proves the assertion. \hfill $\Box$ 

\vspace{3mm}

\noindent\emph{Proof of Proposition \ref{ptpt}.} As usual, we set $\sigma=1$ and $\mu=0$.  Consider the test function $h(x)=\frac{\sin(ax)}{a}$, which is in the class $\mathcal{H}_{\mathrm{W}}$.  Therefore, if there was a general bound of the form $\|f^{(3)}\|\leq M_r\|h'\|$, then we would be able to find a constant $N_r>0$, independent of $a$, such that $\|f^{(3)}\|\leq N_r$. We shall show that $f^{(3)}(x)$ blows up as $x\rightarrow0$ for $a$ such that $ax\ll 1\ll a^2x$, meaning that such a bound cannot be obtained for $\|f^{(3)}\|$ which proves the proposition.  Before performing this analysis, we note that the second derivative $h''(x)=-a\sin(ax)$ blows up if $ax\ll 1\ll a^2 x$ (consider the expansion $\sin(t)=t+O(t^3)$, $t\rightarrow0$).  A bound of the form $\|f^{(3)}\|\leq M_{r,0}\|\tilde{h}\|+M_{r,1}\|h'\|+M_{r,2}\|h''\|$ is therefore still be possible, and we know from Section 3.1.7 of \cite{dgv15} that this is indeed the case.  

 Let $x>0$.  We first obtain a formula for $f^{(3)}(x)$.  To this end, we note that twice differentiating the representation (\ref{vgsolngeneral01}) of the solution and then simplifying using the differentiation formulas (\ref{diff11}) and (\ref{diff22}) followed by the Wronskian formula $I_\nu(x)K_{\nu+1}(x)+I_{\nu+1}(x)K_\nu(x)=\frac{1}{x}$ \cite{olver} gives that
\begin{align*}f''(x) &= \frac{\tilde{h}(x)}{x} -\bigg[\frac{\mathrm{d}^2}{\mathrm{d}x^2}\bigg(\frac{ K_{\nu}(x)}{x^{\nu}}\bigg)\bigg] \int_0^x t^{\nu} I_{\nu}(t)\tilde{h}(t)\,\mathrm{d}t -\bigg[\frac{\mathrm{d}^2}{\mathrm{d}x^2}\bigg(\frac{ I_{\nu }(x)}{x^{\nu}}\bigg)\bigg] \int_x^{\infty} t^{\nu} K_{\nu}(t)\tilde{h}(t)\,\mathrm{d}t.
\end{align*}
Differentiating this formulas then gives
\begin{align}f^{(3)}(x)&=\frac{h'(x)}{x}-\frac{\tilde{h}(x)}{x^2} -\bigg[\frac{\mathrm{d}^3}{\mathrm{d}x^3}\bigg(\frac{ K_{\nu}(x)}{x^{\nu}}\bigg)\bigg] \int_0^x t^{\nu} I_{\nu}(t)\tilde{h}(t)\,\mathrm{d}t+R_1\nonumber \\
&\quad+\tilde{h}(x)\bigg\{-x^{\nu}I_{\nu}(x)\frac{\mathrm{d}^2}{\mathrm{d}x^2}\bigg(\frac{ K_{\nu}(x)}{x^{\nu}}\bigg)+x^{\nu}K_{\nu}(x)\frac{\mathrm{d}^2}{\mathrm{d}x^2}\bigg(\frac{ I_{\nu}(x)}{x^{\nu}}\bigg)\bigg\} \nonumber\\
\label{ff33} &=\frac{h'(x)}{x}-\frac{(2\nu+2)\tilde{h}(x)}{x^2} -\bigg[\frac{\mathrm{d}^3}{\mathrm{d}x^3}\bigg(\frac{K_{\nu}(x)}{x^{\nu}}\bigg)\bigg] \int_0^x t^{\nu} I_{\nu}(t)\tilde{h}(t)\,\mathrm{d}t+R_1,
\end{align} 
where
\begin{align*}R_1= -\bigg[\frac{\mathrm{d}^3}{\mathrm{d}x^3}\bigg(\frac{I_{\nu }(x)}{x^{\nu}}\bigg)\bigg] \int_x^{\infty} t^{\nu} K_{\nu}(t)\tilde{h}(t)\,\mathrm{d}t.
\end{align*}
Here, to obtain equality (\ref{ff33}) we used differentiation formulas (\ref{diff11}) and (\ref{diff22}) followed again by the Wronskian formula.  For all $\nu>-\frac{1}{2}$ and $x>0$, we can use inequalities (\ref{difiineq}) and (\ref{rnmt1}) to bound $R_1$:
\begin{align*}|R_1|\leq \|\tilde{h}\|\bigg[\frac{\mathrm{d}^3}{\mathrm{d}x^3}\bigg(\frac{I_{\nu }(x)}{x^{\nu}}\bigg)\bigg] \int_x^{\infty} t^{\nu} K_{\nu}(t)\,\mathrm{d}t\leq \|\tilde{h}\|\frac{I_{\nu }(x)}{x^{\nu}} \int_x^{\infty} t^{\nu} K_{\nu}(t)\,\mathrm{d}t\leq \|\tilde{h}\|\frac{\sqrt{\pi}\Gamma(\nu+\frac{1}{2})}{2\Gamma(\nu+1)}.
\end{align*}
As $\|\tilde{h}\|\leq 2\|h\|=\frac{2}{a}$, the term $R_1$ does not explode when $a\rightarrow\infty$.



Applying integration by parts to (\ref{ff33}) we obtain
\begin{equation*}f^{(3)}(x)=\frac{h'(x)}{x}+\bigg[\frac{\mathrm{d}^3}{\mathrm{d}x^3}\bigg(\frac{K_{\nu}(x)}{x^{\nu}}\bigg)\bigg]\int_0^xh'(u)\int_0^u t^\nu I_\nu(t)\,\mathrm{d}t\,\mathrm{d}u+R_1+R_2,
\end{equation*}
where
\begin{equation}\label{r2two}R_2=-\tilde{h}(x)\bigg\{\frac{2\nu+2}{x^2}+ \bigg[\frac{\mathrm{d}^3}{\mathrm{d}x^3}\bigg(\frac{K_{\nu}(x)}{x^{\nu}}\bigg)\bigg]\int_0^x t^{\nu} I_{\nu}(t)\,\mathrm{d}t\bigg\}=:-\tilde{h}(x)A_\nu(x).
\end{equation}
For all $\nu>-\frac{1}{2}$, we show that there exists a constant $C_\nu>0$ independent of $x$ such that $A_\nu(x)\leq C_\nu$ for all $x>0$.  To see this, it suffices to consider the behaviour in the limits $x\downarrow0$ and $x\rightarrow\infty$. We first note that $A_\nu(x)\rightarrow0$ as $x\rightarrow\infty$, which follows from using the differentiation formula (\ref{dk3}) followed by (\ref{Ktendinfinity}) and the following limiting form (see \cite{gaunt asym}):
\[\int_0^x t^{\nu} I_{\nu}(t)\,\mathrm{d}t\sim \frac{1}{\sqrt{2\pi}}x^{\nu-1/2}\mathrm{e}^x, \quad x\rightarrow\infty,\:\nu>-\tfrac{1}{2}.\]
Also, using the differentiation formula (\ref{dk3}) followed by the limiting forms (\ref{Itend0}) and (\ref{Ktend0}) gives that, for $\nu>-\frac{1}{2}$, as $x\downarrow0$,
\begin{align*}&\bigg[\frac{\mathrm{d}^3}{\mathrm{d}x^3}\bigg(\frac{K_{\nu}(x)}{x^{\nu}}\bigg)\bigg]\int_0^x t^{\nu} I_{\nu}(t)\,\mathrm{d}t \\
&\quad=-\bigg(\frac{(2\nu+1)K_{\nu}(x)}{x^{\nu}}+\left(1+\frac{(2\nu+1)(2\nu+2)}{x^2}\right)\frac{K_{\nu+1}(x)}{x^{\nu}}\bigg)\int_0^x t^{\nu} I_{\nu}(t)\,\mathrm{d}t \\
&\quad=-\bigg(\frac{(2\nu+1)(2\nu+2)\cdot 2^\nu\Gamma(\nu+1)}{x^{2\nu+3}}+O(x^{-2\nu-1})\bigg)\int_0^x \frac{t^{2\nu}}{2^\nu\Gamma(\nu+1)}\,\mathrm{d}t \\
&\quad=-\frac{2\nu+2}{x^2}+O(1),
\end{align*}
and therefore $A_\nu(x)$ is bounded as $x\downarrow0$, as required.  We conclude that $R_2$ does not explode when $a\rightarrow\infty$.


Now, we use the differentiation formula (\ref{dk3}) to obtain
\begin{align*}f^{(3)}(x)&=\frac{h'(x)}{x}-\frac{(2\nu+1)(2\nu+2)K_{\nu+1}(x)}{x^{\nu+2}}\int_0^xh'(u)\int_0^u t^\nu I_\nu(t)\,\mathrm{d}t\,\mathrm{d}u\\
&\quad+R_1+R_2+R_3,
\end{align*} 
where
\begin{align}|R_3|&=\bigg|\bigg(\frac{(2\nu+1)K_\nu(x)}{x^\nu}+\frac{K_{\nu+1}(x)}{x^\nu}\bigg)\int_0^xh'(u)\int_0^u t^\nu I_\nu(t)\,\mathrm{d}t\,\mathrm{d}u\bigg|\nonumber \\
\label{r3three}&\leq \bigg(\frac{(2\nu+1)K_\nu(x)}{x^\nu}+\frac{K_{\nu+1}(x)}{x^\nu}\bigg)\cdot \frac{2(\nu+2)}{2\nu+1}x^\nu I_{\nu+2}(x),
\end{align}
where we used (\ref{gaunt14}) and that $\|h'\|=1$ to obtain the second inequality.  For $\nu>-\frac{1}{2}$, the expression involving modified Bessel functions in (\ref{r3three}) is uniformly bounded for all $x\geq0$, which can be seen from a straightforward analysis involving the asymptotic formulas (\ref{Itend0}) -- (\ref{Ktendinfinity}).  Therefore, the term $R_3$ does not explode when $a\rightarrow\infty$.

We now analyse the behaviour of $f^{(3)}(x)$ in a neighbourhood of $x=0$ when $a\rightarrow\infty$.  For all $x\geq0$, the terms $R_1$, $R_2$ and $R_3$ are $O(1)$ as $a\rightarrow\infty$.  Therefore using the asymptotic formulas (\ref{Itend0}) and (\ref{Ktend0}) we obtain
\begin{align*}f^{(3)}(x)&=-\frac{\cos(ax)}{x}+\frac{(2\nu+1)(2\nu+2)}{x^{\nu+2}}\cdot\frac{2^\nu \Gamma(\nu+1)}{x^{\nu+1}}\times \\
&\quad \times\int_0^x \cos(au)\int_0^u\frac{t^{2\nu}}{2^\nu\Gamma(\nu+1)}\,\mathrm{d}t\,\mathrm{d}u+O(1) \\
&=-\frac{\cos(ax)}{x}+\frac{(2\nu+1)(2\nu+2)}{x^{2\nu+3}}\!\int_0^x\! u^{2\nu+1}\cos(au)\,\mathrm{d}u+O(1), \:\: x\downarrow 0.
\end{align*}
In addition to $x\downarrow0$ and $a\rightarrow\infty$ we let $ax\downarrow0$.  Therefore on using that $\cos(t)=1-\frac{1}{2}t^2+O(t^4)$ as $t\downarrow0$, we have that, in this regime,
\begin{align*}f^{(3)}(x)&=-\frac{1}{x}\bigg(1-\frac{a^2x^2}{2}\bigg)+\frac{2\nu+2}{x^{2\nu+3}}\int_0^x u^{2\nu+1}\bigg(1-\frac{a^2u^2}{2}\bigg)\,\mathrm{d}u+O(1) \\
&=\frac{a^2x}{2}-\frac{(\nu+1)a^2x}{2\nu+4}+O(1) =\frac{a^2x}{2(\nu+2)}+O(1).
\end{align*}
If we take choose $a$ such that $ax\ll1\ll a^2x$, then $f^{(3)}(x)$ blows up in a neighbourhood of the origin, which proves the assertion. \hfill $\Box$

\vspace{3mm}

\noindent\emph{Proof of (\ref{sec6def}).} As usual, we set $\sigma=1$.  From the formula (\ref{vgsolngeneral01}) for the solution of the $\mathrm{SVG}(r,1,0)$ Stein equation we have
\begin{align*}\lim_{x\rightarrow\infty}xf(x)&=-\lim_{x\rightarrow\infty}\bigg\{\frac{K_{\nu}(x)}{x^\nu}\int_0^x t^\nu I_\nu(t)\tilde{h}(t)\,\mathrm{d}t -\frac{I_{\nu}(x)}{x^\nu}\int_x^\infty t^\nu K_\nu(t)\tilde{h}(t)\,\mathrm{d}t\bigg\}\\
&=:I_1+I_2.
\end{align*}
We shall use L'H\^{o}pital's rule to calculate $I_1$ and $I_2$.  In anticipation of this we note that
\[\frac{\mathrm{d}}{\mathrm{d}x}\bigg(\frac{x^{\nu-1}}{K_\nu(x)}\bigg)=\frac{\mathrm{d}}{\mathrm{d}x}\bigg(\frac{1}{x}\bigg/\frac{K_\nu(x)}{x^\nu}\bigg)=-\frac{x^{\nu-2}}{K_\nu(x)}+\frac{x^{\nu-1}K_{\nu+1}(x)}{K_\nu(x)^2},
\]
where we used the quotient rule and (\ref{diff22}) in the final step.  Similarly, on using (\ref{diff11}) we obtain 
\[\frac{\mathrm{d}}{\mathrm{d}x}\bigg(\frac{x^{\nu-1}}{I_\nu(x)}\bigg)=-\frac{x^{\nu-2}}{I_\nu(x)}-\frac{x^{\nu-1}I_{\nu+1}(x)}{I_\nu(x)^2}.\]
Therefore, by L'H\^{o}pital's rule,
\begin{align*}I_1&=-\lim_{x\rightarrow\infty}\left\{\frac{x^\nu I_\nu(x)\tilde{h}(x)}{\displaystyle-\frac{x^{\nu-2}}{K_\nu(x)}+\frac{x^{\nu-1}K_{\nu+1}(x)}{K_\nu(x)^2}}\right\}=-\frac{1}{2}\tilde{h}(\infty), \\
I_2&=-\lim_{x\rightarrow\infty}\left\{\frac{-x^\nu K_\nu(x)\tilde{h}(x)}{\displaystyle-\frac{x^{\nu-2}}{I_\nu(x)}-\frac{x^{\nu-1}I_{\nu+1}(x)}{I_\nu(x)^2}}\right\}=-\frac{1}{2}\tilde{h}(\infty),
\end{align*}
where we used the asymptotic formulas (\ref{roots}) and (\ref{Ktendinfinity}) to compute the limits.  Thus, $\lim_{x\rightarrow\infty}xf(x)=-\tilde{h}(\infty)$.  Similarly, by considering (\ref{vgsolngeneral11}) instead of (\ref{vgsolngeneral01}), we obtain $\lim_{x\rightarrow-\infty}xf(x)=\tilde{h}(-\infty)$. \hfill $\Box$

\vspace{3mm}

The following lemma will be used in the proof of Proposition \ref{prop1}.

\begin{lemma}\label{beslem}(i) Let $\nu>0$.  Then $x^\nu K_\nu(x)\leq 2^{\nu-1}\Gamma(\nu)$ for all $x>0$.

(ii) Suppose $0<x<0.729$.  Then $K_0(x)<-2\log(x)$.
\end{lemma}

\begin{proof}(i) We have that $\frac{\mathrm{d}}{\mathrm{d}x}\big(x^\nu K_\nu(x)\big)=-x^\nu K_{\nu-1}(x)<0$ (see (\ref{ddbk})), which implies that $x^\nu K_\nu(x)$ is a decreasing function of $x$.  From (\ref{Ktend0}) we have $\lim_{x\downarrow 0}x^\nu K_\nu(x)=2^{\nu-1}\Gamma(\nu)$, and we thus deduce the inequality. 

(ii) From the differentiation formula (\ref{dkzero}), for all $x>0$, $\frac{\mathrm{d}}{\mathrm{d}x}\big(-2\log(x)-K_0(x)\big)=-\frac{2}{x}+K_1(x)<0$, where the inequality follows from part (i).  Therefore $-2\log(x)-K_0(x)$ is a decreasing of function of $x$.  But one can check numerically using Mathematica that $-2\log(0.729)-K_0(0.729)=0.00121$, and the conclusion follows.
\end{proof}

\noindent\emph{Proof of Proposition \ref{prop1}.} For ease of notation, we shall set $\mu=0$; the extension to general $\mu\in\mathbb{R}$ is obvious.  Throughout this proof, $Z$ will denote a $\mathrm{SVG}(r,\sigma,0)$ random variable.


(i) Let $r>1$.  Proposition 1.2 of \cite{ross} states that if a random variable $Y$ has Lebesgue density bounded by $C$, then for any random variable $W$,
\begin{equation}\label{subpr1}d_{\mathrm{K}}(W,Y)\leq\sqrt{2Cd_{\mathrm{W}}(W,Y)}.
\end{equation} 
Since the $\mathrm{SVG}(r,\sigma,0)$ distribution is unimodal about $0$, it follows from (\ref{pmutend}) that the density is bounded above by $C=\frac{1}{2\sigma\sqrt{\pi}}\Gamma(\frac{r-1}{2})/\Gamma(\frac{r}{2})$, which on substituting into (\ref{subpr1}) yields the desired bound.

(ii) Here we consider the case $r=1$. We begin by following the approach used in the proof of Proposition 1.2 of \cite{ross}, but we need to alter the argument because the $\mathrm{SVG}(1,\sigma,0)$ density $p(x)=\frac{1}{\pi\sigma}K_0\big(\frac{|x|}{\sigma}\big)$ is unbounded as $x\rightarrow0$.  Consider the functions $h_z(x)=\mathbf{1}(x\leq z)$, and the `smoothed' $h_{z,\alpha}(x)$ defined to be one for $x\leq z+2\alpha$, zero for $x>z$, and linear between.  Then
\begin{align}\mathbb{P}(W\leq z)-\mathbb{P}(Z\leq z)&=\mathbb{E}h_{z}(W)-\mathbb{E}h_{z,\alpha}(Z)+\mathbb{E}h_{z,\alpha}(Z)-\mathbb{E}h_z(Z)\nonumber \\
&\leq \mathbb{E}h_{z,\alpha}(W)-\mathbb{E}h_{z,\alpha}(Z)+\frac{1}{2}\mathbb{P}(z\leq Z\leq z+2\alpha)\nonumber \\
&\leq \frac{1}{2\alpha}d_{\mathrm{W}}(W,Z)+\frac{1}{2}\mathbb{P}(z\leq Z\leq z+2\alpha)\nonumber \\
\label{puting}&\leq \frac{1}{2\alpha}d_{\mathrm{W}}(W,Z)+\mathbb{P}(0\leq Z\leq \alpha),
\end{align}
where the last inequality follows because the $\mathrm{SVG}(1,\sigma,0)$ density is a decreasing function of $x$ for $x>0$ and an increasing function for $x<0$, and so $\mathbb{P}(z\leq Z\leq z+2\alpha)$ is maximised for $z=-\alpha$.  Suppose that $\frac{\alpha}{\sigma}<0.729$.  Then we can use Lemma \ref{beslem} to obtain
\begin{align}\mathbb{P}(0\leq Z<\alpha)&=\int_0^\alpha \frac{1}{\pi\sigma} K_0\bigg(\frac{t}{\sigma}\bigg)\,\mathrm{d}t =\frac{1}{\pi}\int_0^{\frac{\alpha}{\sigma}}K_0(y)\,\mathrm{d}y \nonumber\\
\label{uu2}&\leq \frac{1}{\pi}\int_0^{\frac{\alpha}{\sigma}}-2\log(y)\,\mathrm{d}y =\frac{2\alpha}{\pi\sigma}\bigg[1+\log\bigg(\frac{\sigma}{\alpha}\bigg)\bigg].
\end{align}  
Substituting into (\ref{puting}) gives that, for any $z\in\mathbb{R}$ and $\alpha>0$,
\begin{equation*}\mathbb{P}(W\leq z)-\mathbb{P}(Z\leq z)\leq \frac{1}{2\alpha}d_{\mathrm{W}}(W,Z)+\frac{2\alpha}{\pi\sigma}\bigg[1+\log\bigg(\frac{\sigma}{\alpha}\bigg)\bigg].
\end{equation*}
We take $\alpha=\frac{1}{2}\sqrt{\pi\sigma d_{\mathrm{W}}(W,Z)}$, which, as we assumed that $\sigma^{-1}d_{\mathrm{W}}(W,Z)<0.676$, ensures that $\frac{\alpha}{\sigma}<0.729$.  This leads to the upper bound
\begin{equation*}\mathbb{P}(W\leq z)-\mathbb{P}(Z\leq z)\leq\bigg\{2+\log\bigg(\frac{2}{\sqrt{\pi}}\bigg)+\frac{1}{2}\log\bigg(\frac{\sigma}{d_{\mathrm{W}}(W,Z)}\bigg)\bigg\}\sqrt{\frac{d_{\mathrm{W}}(W,Z)}{\pi\sigma}}.
\end{equation*}
Similarly, we can show that
\begin{equation*}\mathbb{P}(W\leq z)-\mathbb{P}(Z\leq z)\geq-\bigg\{2+\log\bigg(\frac{2}{\sqrt{\pi}}\bigg)+\frac{1}{2}\log\bigg(\frac{\sigma}{d_{\mathrm{W}}(W,Z)}\bigg)\bigg\}\sqrt{\frac{d_{\mathrm{W}}(W,Z)}{\pi\sigma}}.
\end{equation*}
Combining these bounds proves (\ref{pronf2}).

(iii) Let $0<r<1$. Then the $\mathrm{SVG}(r,\sigma,0)$ density is unbounded as $x\rightarrow0$ and is a decreasing function of $x$ for $x>0$ and an increasing function for $x<0$.  Therefore we argue as we did in part (ii) and bound $\mathbb{P}(0\leq Z\leq \alpha)$ and then substitute into (\ref{puting}).  Let $\nu=\frac{r-1}{2}$, so that $-\frac{1}{2}<\nu<0$.  We have
\begin{align}\mathbb{P}(0\leq Z\leq \alpha)&=\frac{1}{\sigma\sqrt{\pi}2^\nu\Gamma(\nu+\frac{1}{2})}\int_0^\alpha \bigg(\frac{t}{\sigma}\bigg)^\nu K_\nu\bigg(\frac{t}{\sigma}\bigg)\,\mathrm{d}t\nonumber \\
&=\frac{1}{\sqrt{\pi}2^\nu\Gamma(\nu+\frac{1}{2})}\int_0^{\frac{\alpha}{\sigma}} y^{2\nu}\cdot y^{-\nu} K_{-\nu}(y) \,\mathrm{d}y\nonumber \\
&\leq\frac{1}{\sqrt{\pi}2^\nu\Gamma(\nu+\frac{1}{2})}\int_0^{\frac{\alpha}{\sigma}}2^{-\nu-1}\Gamma(-\nu) y^{2\nu} \,\mathrm{d}y\nonumber \\
\label{uu3}&=\frac{\Gamma(-\nu)}{\sqrt{\pi}2^{2\nu+1}\Gamma(\nu+\frac{1}{2})}\frac{1}{2\nu+1}\bigg(\frac{\alpha}{\sigma}\bigg)^{2\nu+1}=C_{\nu,\sigma}\alpha^{2\nu+1},
\end{align}
where we used a change of variables and (\ref{kpart}) in the second step and Lemma \ref{beslem} in the third.  We therefore have that, for any $z\in\mathbb{R}$ and $\alpha>0$,
\begin{equation*}\mathbb{P}(W\leq z)-\mathbb{P}(Z\leq z)\leq \frac{1}{2\alpha}d_{\mathrm{W}}(W,Z)+C_{\nu,\sigma}\alpha^{2\nu+1}.
\end{equation*}
To optimise, we take $\alpha=\big(\frac{d_{\mathrm{W}}(W,Z)}{2(2\nu+1)C_{\nu,\sigma}}\big)^{\frac{1}{2(\nu+1)}}$, which results in the bound
\begin{align*}\mathbb{P}(W\leq z)-\mathbb{P}(Z\leq z)&\leq 2\big(2(2\nu+1)C_{\nu,\sigma}\big)^{\frac{1}{2(\nu+1)}}\big(d_{\mathrm{W}}(W,Z)\big)^{\frac{2\nu+1}{2(\nu+1)}}\\
&=2\bigg(\frac{2\Gamma(-\nu)}{\sqrt{\pi}(2\sigma)^{2\nu+1}\Gamma(\nu+\frac{1}{2})}\bigg)^{\frac{1}{2(\nu+1)}}\big(d_{\mathrm{W}}(W,Z)\big)^{\frac{2\nu+1}{2(\nu+1)}}.
\end{align*}
As in part (ii), we can similarly obtain a lower bound, and on substituting $\nu=\frac{r-1}{2}$ we obtain (\ref{pronf3}), which completes the proof.  \hfill $\Box$ 

\appendix

\section{Properties of modified Bessel functions}\label{appendix}

Here we list standard properties and inequalities for modified Bessel functions that are used throughout this paper.  All formulas can be found in \cite{olver}, except for the differentation fromulas (\ref{2ndii})--(\ref{dk3}), which can be found in \cite{gaunt thesis} and \cite{gaunt diff}, and the inequalities.

The \emph{modified Bessel function of the first kind} of order $\nu \in \mathbb{R}$ is defined, for $x\in\mathbb{R}$, by
\begin{equation*}\label{defI}I_{\nu} (x) = \sum_{k=0}^{\infty} \frac{1}{\Gamma(\nu +k+1) k!} \left( \frac{x}{2} \right)^{\nu +2k}.
\end{equation*}
The \emph{modified Bessel function of the second kind} of order $\nu\in\mathbb{R}$ is defined, for $x>0$, by
\begin{equation}\label{parity}K_\nu(x)=\int_0^\infty \mathrm{e}^{-x\cosh(t)}\cosh(\nu t)\,\mathrm{d}t.
\end{equation}
It is clear from (\ref{parity}) that
\begin{equation}\label{kpart}K_{-\nu}(x)=K_\nu(x).
\end{equation}
The modified Bessel functions have the following asymptotic behaviour:
\begin{eqnarray}\label{Itend0}I_{\nu} (x) &\sim& \frac{1}{\Gamma(\nu +1)} \left(\frac{x}{2}\right)^{\nu}\big(1+O(x^2)\big), \qquad x \downarrow 0, \\
\label{Ktend0}K_{\nu} (x) &\sim& \begin{cases} 2^{|\nu| -1} \Gamma (|\nu|) x^{-|\nu|}\big(1+O(x^2)\big), & \quad x \downarrow 0, \: \nu \not= 0, \\
-\log x, & \quad x \downarrow 0, \: \nu = 0, \end{cases} \\
 \label{roots} I_{\nu} (x) &\sim& \frac{\mathrm{e}^x}{\sqrt{2\pi x}}, \quad x \rightarrow \infty,  \\
\label{Ktendinfinity} K_{\nu} (x) &\sim& \sqrt{\frac{\pi}{2x}} \mathrm{e}^{-x}, \quad x \rightarrow \infty.
\end{eqnarray}
The following differentiation formulas hold:
\begin{eqnarray}
\label{dkzero}\frac{\mathrm{d}}{\mathrm{d}x}\big(K_0(x)\big)&=&-K_{1}(x), \\
\label{ddbk}\frac{\mathrm{d}}{\mathrm{d}x}\big(x^\nu K_\nu(x)\big)&=&-x^{\nu} K_{\nu-1}(x), \\
\label{diff11}\frac{\mathrm{d}}{\mathrm{d}x}\bigg(\frac{I_\nu(x)}{x^\nu}\bigg)&=&\frac{I_{\nu+1}(x)}{x^\nu}, \\
\label{diff22}\frac{\mathrm{d}}{\mathrm{d}x}\bigg(\frac{K_\nu(x)}{x^\nu}\bigg)&=&-\frac{K_{\nu+1}(x)}{x^\nu}, \\
\label{2ndkk}\frac{\mathrm{d}^2}{\mathrm{d}x^2}\left(\frac{I_{\nu}(x)}{x^{\nu}}\right)&=&\frac{I_{\nu}(x)}{x^{\nu}}-\frac{(2\nu+1)I_{\nu+1}(x)}{x^{\nu+1}},  \\
\label{2ndii}\frac{\mathrm{d}^2}{\mathrm{d}x^2}\left(\frac{K_{\nu}(x)}{x^{\nu}}\right)&=&\frac{K_{\nu}(x)}{x^{\nu}}+\frac{(2\nu+1)K_{\nu+1}(x)}{x^{\nu+1}}, \\
\label{dk3}\frac{\mathrm{d}^3}{\mathrm{d}x^3}\left(\frac{K_{\nu}(x)}{x^{\nu}}\right)&=&-\frac{(2\nu+1)K_{\nu}(x)}{x^{\nu}}-\left(1+\frac{(2\nu+1)(2\nu+2)}{x^2}\right)\frac{K_{\nu+1}(x)}{x^{\nu}}.
\end{eqnarray}
Applying the inequality $I_{\mu+1}(x)<I_{\mu}(x)$, $x>0$, $\mu>-\frac{1}{2}$ \cite{soni} to the sixth differentiation formula of Corollary 1 of \cite{gaunt diff} gives the inequality
\begin{equation}\label{difiineq}\frac{\mathrm{d}^3}{\mathrm{d}x^3}\left(\frac{I_{\nu}(x)}{x^{\nu}}\right)< \frac{I_{\nu}(x)}{x^\nu}, \quad x>0,\:\nu>-\tfrac{1}{2}.
\end{equation}
The next inequality follows from two applications of inequality (2.6) of \cite{gaunt ineq1}.  For $x\geq0$,
\begin{equation}\label{gaunt14}\int_0^x\!\int_0^u  t^\nu I_\nu(t)\,\mathrm{d}t\,\mathrm{d}u\leq \frac{2(\nu+2)}{2\nu+1}x^\nu I_{\nu+2}(x), \quad \nu>-\tfrac{1}{2}.
\end{equation}
The following bounds, which can be found in \cite{gaunt ineq2, gaunt ineq3}, are used to bound the solution to the SVG Stein equation.  Let $\nu>-\frac{1}{2}$.  Then, for all $x\geq 0$,
\begin{eqnarray}\label{propb2a12}\label{propb2a125}\frac{K_{\nu}(x)}{x^\nu}\int_0^x t^{\nu+1}I_\nu(t)\,\mathrm{d}t&<& \frac{1}{2}, \\
\label{fff1}\frac{I_{\nu}(x)}{x^\nu}\int_x^\infty t^{\nu+1}K_{\nu}(t)\,\mathrm{d}t&<& 1, \\
\label{rnmt2}\frac{K_\nu(x)}{x^\nu}\int_0^x t^\nu I_\nu(t)\,\mathrm{d}t&\leq& \frac{1}{2\nu+1},\\
\label{rnmt1}\frac{I_\nu(x)}{x^\nu}\int_x^\infty t^\nu K_\nu(t)\,\mathrm{d}t&\leq& \frac{\sqrt{\pi}\Gamma(\nu+\frac{1}{2})}{2\Gamma(\nu+1)},  \\
\label{jjj1}\frac{K_{\nu}(x)}{x^{\nu-1}}\int_0^x t^{\nu}I_\nu(t)\,\mathrm{d}t&<& \frac{\nu+1}{2\nu+1}, \\
\label{ddd2}\frac{I_{\nu}(x)}{x^{\nu-1}}\int_x^\infty t^{\nu}K_{\nu}(t)\,\mathrm{d}t&<& 1, \\
\label{jjj1z}\frac{K_{\nu+1}(x)}{x^{\nu-1}}\int_0^x t^{\nu}I_\nu(t)\,\mathrm{d}t&<& \frac{\nu+1}{2\nu+1}, \\
\label{ddd2z}\frac{I_{\nu+1}(x)}{x^{\nu-1}}\int_x^\infty t^{\nu}K_{\nu}(t)\,\mathrm{d}t&<& \frac{1}{2}.
\end{eqnarray}

\section*{Acknowledgements}
The author is supported by a Dame Kathleen Ollerenshaw Research Fellowship. The author would like to thank Ivan Nourdin for a helpful discussion.  The author would also like to thank the referee for a careful reading of the manuscript and their helpful comments.


\footnotesize

\begin{thebibliography}{99}

\bibitem{aaps17} Arras, B., Azmoodeh, E., Poly, G. and Swan, Y. A bound on the 2-Wasserstein distance between linear combinations of independent random variables. To appear in \emph{Stoch. Proc. Appl.}, 2018+.

\bibitem{bt17} Bai, S. and Taqqu, M.  Behavior of the generalized Rosenblatt process at extreme critical exponent values. \emph{Ann. Probab.} $\mathbf{45}$ (2017), pp. 1278--1324.

\bibitem{bhj92} Barbour, A. D., Holst, L. and Janson, S. \emph{Poisson Approximation}. Oxford University Press, Oxford, 1992.

\bibitem{bla} Blaisdell, B. A measure of the similarity of sets of sequences not requiring sequence alignment. \emph{Proc. Natl. Acad. Sci. USA} $\mathbf{83}$ (1986), pp. 5155--5159.

\bibitem{chatterjee} Chatterjee, S., Fulman, J. and  R\"ollin, A. Exponential approximation by Stein's method and spectral graph theory.
\emph{ALEA Lat. Am. J. Probab. Math. Stat.} $\mathbf{8}$ (2011), pp. 197--223.

\bibitem{chen} Chen, L. H. Y., Goldstein, L. and Shao, Q.--M.  \emph{Normal Approximation by Stein's Method.} Springer, 2011.

\bibitem{cr10} Chen, L. H. Y. and R\"{o}llin, A. Stein couplings for normal approximation. arXiv:1003:6039, 2010.

\bibitem{dobler} D\"{o}bler, C. Distributional transformations without orthogonality relations. \emph{J. Theoret. Probab.} $\mathbf{30}$ (2017), pp. 85--116.

\bibitem{dgv15} D\"{o}bler, C, Gaunt, R. E. and Vollmer, S. J.  An iterative technique for bounding derivatives of solutions of Stein equations.  \emph{Electron. J. Probab.} $\mathbf{22}$ no. 96 (2017), pp. 1--39.

\bibitem{eberlein} Eberlein, E. and Hammerstein E.  Generalized Hyperbolic and Inverse Gaussian Distributions: Limiting Cases and Approximation of Processes.  In:  Dalang, R. C. Dozzi, M., Russo, F. (Eds.), Seminar on Stochastic Analysis, Random
Fields and Applications IV, in: \emph{Progress in Probability} $\mathbf{58}$ Birkh\"{a}user Verlag, (2004), pp. 105--153.

\bibitem{eden2} Eden, R. and Viquez, J. Nourdin-Peccati analysis on Wiener and Wiener-Poisson space for general distributions. \emph{Stoch. Proc. Appl.} $\mathbf{125}$ (2015), pp. 182--216.

\bibitem{eichelsbacher} Eichelsbacher, P. and Th\"{a}le, C.  Malliavin-Stein method for Variance-Gamma approximation on Wiener space.  	\emph{Electron. J. Probab.} $\mathbf{20}$ no. 123 (2015), pp. 1--28.

\bibitem{elezovic} Elezovi\'c, N., Giordano, C. and Pe\v{c}ari\'c, J.  The best bounds in Gautschi's inequality. \emph{Math. Inequal. Appl.} $\mathbf{3}$ (2000), pp. 239--252.

\bibitem{finlay} Finlay, R. and Seneta, E.  Option pricing with VG-like models. \emph{Int. J. Theor. Appl. Finan.} $\mathbf{11}$ (2008), pp. 943--955.

\bibitem{gaunt thesis} Gaunt, R. E.  \emph{Rates of Convergence of Variance-Gamma Approximations via Stein's Method.}  DPhil thesis, University of Oxford, 2013.

\bibitem{gaunt vg} Gaunt, R. E.  Variance-Gamma approximation via Stein's method.  \emph{Electron. J. Probab.} $\mathbf{19}$ no. 38 (2014), pp. 1--33.

\bibitem{gaunt ineq1} Gaunt, R. E.  Inequalities for modified Bessel functions and their integrals.  \emph{J. Math. Anal. Appl.} $\mathbf{420}$ (2014), pp. 373--386.

\bibitem{gaunt ineq2} Gaunt, R. E.  Uniform bounds for expressions involving modified Bessel functions.  \emph{Math. Inequal. Appl.} $\mathbf{19}$ (2016), pp. 1003--1012.

\bibitem{gaunt pn} Gaunt, R. E. On Stein's method for products of normal random variables and zero bias couplings.  \emph{Bernoulli} $\mathbf{23}$ (2017), pp. 3311--3345.

\bibitem{gaunt diff} Gaunt, R. E. Derivative formulas for Bessel, Struve and Anger-Weber functions. \emph{J. Classical Anal.} $\mathbf{11}$ (2017), pp. 69-78.

\bibitem{gaunt ineq3} Gaunt, R. E. Inequalities for integrals of modified Bessel functions and expressions involving them. \emph{J. Math. Anal. Appl.} $\mathbf{462}$ (2018), pp. 172--190.

\bibitem{gaunt asym} Gaunt, R. E. Inequalities for some integrals involving modified Bessel functions. To appear in \emph{P. Am. Math. Soc.}, 2018+.

\bibitem{gaunt vgii} Gaunt, R. E. Wasserstein and Kolmogorov error bounds for variance-gamma approximation via Stein's method II. \emph{In preparation}, 2018+.

\bibitem{gaunt chi square} Gaunt, R. E., Pickett, A. M. and Reinert, G.  Chi-square approximation by Stein's method with application to Pearson's statistic.  \emph{Ann. Appl. Probab.} $\mathbf{27}$ (2017), pp. 720--756.

\bibitem{gaut} Gautschi, W. Some elementary inequalities relating to the gamma and incomplete gamma function. \emph{J. Math. Phys.} $\mathbf{38}$ (1959), pp. 77--81.

\bibitem{gio} Giordanoa, C. and Laforgia, A.  Inequalities and monotonicity properties for the gamma function. \emph{J. Comput. Appl. Math.} $\mathbf{133}$ (2001), pp. 387--396.

\bibitem{goldstein} Goldstein, L. and Reinert, G.  Stein's Method and the zero bias transformation with application to simple random sampling.  \emph{Ann. Appl. Probab.} $\mathbf{7}$ (1997), pp. 935--952.

\bibitem{k97} Kalashnikov, V. \emph{Geometric Sums: Bounds for Rare Events with Applications. Risk Analysis, Reliability, Queueing.}  Kluwer Academic Publishers Group, Dordrecht, 1997.

\bibitem{kkp01} Kotz, S., Kozubowski, T. J. and Podg\'{o}rski, K. \emph{The Laplace Distribution and Generalizations: A Revisit with New Applications.} Springer, 2001. 

\bibitem{ley} Ley, C., Reinert, G. and Swan, Y.  Stein's method for comparison of univariate distributions. \emph{Probab. Surv.} $\mathbf{14}$ (2017), pp. 1--52.

\bibitem{lippert} Lippert, R. A., Huang, H. and Waterman, M. S.  Distributional regimes for the number of $k$-word
matches between two random sequences.  \emph{P. Natl. Acad. Sci. USA}  $\mathbf{99}$ (2002), pp. 13980--13989.  

\bibitem{madan} Madan, D. B. and Seneta, E. The Variance Gamma (V.G.) Model for Share Market Returns.  \emph{J. Bus.} $\mathbf{63}$ (1990), pp. 511--524.

\bibitem{madan2} Madan, D. B., Carr, P. and Chang, E. C.  The variance gamma process and option pricing. \emph{Eur. Finance Rev.} $\mathbf{2}$ (1998), pp. 74--105.

\bibitem{np09} Nourdin, I. and Peccati, G. Stein's method on Wiener chaos.  \emph{Probab. Theory Rel.} $\mathbf{145}$ (2009), pp. 75--118.

\bibitem{np10} Nourdin, I. and Peccati, G. Cumulants on the Wiener space. \emph{J. Funct. Anal.} $\mathbf{258}$ (2010), pp. 3775--3791.

\bibitem{np12} Nourdin, I. and Peccati, G. Normal approximations with Malliavin calculus: from Stein's method to universality. Vol. 192. Cambridge University Press, 2012. 

\bibitem{np15} Nourdin, I. and Peccati, G. The optimal fourth moment theorem. \emph{P. Am. Math. Soc.} $\mathbf{143}$ (2015), pp. 3123--3133.

\bibitem{olver} Olver, F. W. J., Lozier, D. W., Boisvert, R. F. and Clark, C. W.  \emph{NIST Handbook of Mathematical Functions.} Cambridge University Press, 2010.

\bibitem{pekoz1} Pek\"oz, E. and R\"ollin, A. New rates for exponential approximation and the theorems of R\'{e}nyi and Yaglom. \emph{Ann. Probab.} $\mathbf{39}$ (2011), pp. 587--608.

\bibitem{prr13} Pek\"oz, E., R\"ollin, A. and Ross, N.  Total variation error bounds for geometric approximation. \emph{Bernoulli} $\mathbf{19}$ (2013), pp. 610--632.

\bibitem{pekoz} Pek\"oz, E., R\"ollin, A. and Ross, N. Degree asymptotics with rates for preferential attachment random graphs. \emph{Ann. Appl. Probab.} $\mathbf{23}$ (2013), pp. 1188--1218.

\bibitem{prr16} Pek\"oz, E., R\"ollin, A. and Ross, N. Generalized gamma approximation with rates for urns, walks and trees. \emph{Ann. Probab.} $\mathbf{44}$ (2016), pp. 1776--1816.

\bibitem{pike} Pike, J. and Ren, H. Stein's method and the Laplace distribution. \emph{ALEA Lat. Am. J. Probab. Math. Stat.} $\mathbf{11}$ (2014), pp. 571-587.

\bibitem{pr12} Pitman, J. and Ross, N. Archimedes, Gauss, and Stein. \emph{Not. Am. Math. Soc.} $\mathbf{59}$ (2012), pp. 1416--1421. 

\bibitem{waterman} Reinert, G., Chew, D., Sun, F. and Waterman, M. S.  Alignment free sequence comparison (I): statistics and power.  \emph{J. Comput. Biol.} $\mathbf{16}$ (2009),  pp. 1615--1634.

\bibitem{lothaire} Reinert, G., Schbath, S. and Waterman, M. S. in Lothaire, M.  \emph{Applied Combinatorics on Words.}  Cambridge University Press, 2005.

\bibitem{renyi} R\'{e}nyi, A. A characterization of Poisson processes. \emph{Magyar Tud. Akad. Mat. Kutat\'{o} Int. K\"{o}zl.} $\mathbf{1}$ (1957), pp. 519--527.

\bibitem{ross} Ross, N. Fundamentals of Stein's method.  \emph{Probab. Surv.} $\mathbf{8}$ (2011), pp. 210--293.

\bibitem{soni} Soni, R. P. On an inequality for modified Bessel functions. \emph{J. Math. Phys. Camb.} $\mathbf{44}$ (1965), pp. 406--407.

\bibitem{stein} Stein, C.  A bound for the error in the normal approximation to the the distribution of a sum of dependent random variables.  In \emph{Proc. Sixth Berkeley Symp. Math. Statis. Prob.} (1972), vol. 2, Univ. California Press, Berkeley, pp. 583--602.

\bibitem{stein2} Stein, C.  \emph{Approximate Computation of Expectations.} IMS, Hayward, California, 1986.

\bibitem{toda} Toda, A. A. Weak limit of the geometric sum of independent but not identically distributed random variables. arXiv:1111.1786, 2011.

\end{thebibliography}
\end{document}